\documentclass[t12pt,twoside,psfig,epsf]{amsart}
\usepackage{amsmath,amssymb}

\usepackage{amsfonts}
\usepackage{amsthm,amscd}
\usepackage{graphicx}
\usepackage{amsmath}
\usepackage{amssymb}
\usepackage{amstext}
\usepackage{color}

\begin{document}

\newcommand{\co}{\mathbb{C}}
\newcommand{\incl}[1]{i_{U_{#1}-Q_{#1},V_{#1}-P_{#1}}}
\newcommand{\inclu}[1]{i_{V_{#1}-P_{#1},\co-P}}
\newcommand{\func}[3]{#1\colon #2 \rightarrow #3}
\newcommand{\norm}[1]{\left\lVert#1\right\rVert}
\newcommand{\norma}[2]{\left\lVert#1\right\rVert_{#2}}
\newcommand{\hiper}[3]{\left \lVert#1\right\rVert_{#2,#3}}
\newcommand{\hip}[2]{\left \lVert#1\right\rVert_{U_{#2} - Q_{#2},V_{#2} - P_{#2}}}
\newtheorem{lem}{Lemma}[section]
\newtheorem{defi}{Definition}[section]
\newtheorem{rem}{Remark}[section]
\newtheorem{mle}{Main Lemma}[section]
\newtheorem{mth}{Main Theorem}
\newtheorem{thm}{Theorem}[section]
\newtheorem{theorem}{Theorem}[section]
\newtheorem{lemma}[theorem]{Lemma}
\newtheorem{e-proposition}[theorem]{Proposition}
\newtheorem{corollary}[theorem]{Corollary}
\newtheorem{e-definition}[theorem]{Definition\rm}
\newtheorem{remark}{\it Remark\/}
\newtheorem{example}{Example}

\newtheorem{cor}{Corollary}[section]
\newtheorem{pro}{Proposition}[section]
\newtheorem{conj}{Conjecture}

\topmargin -5.4cm \headheight 0cm \headsep 6.0cm \textheight 22.2cm
\textwidth 16.0cm \oddsidemargin 0.7cm \evensidemargin 0.7cm
\footskip 1.5cm



\def\A{{\mathbb A}}
\def\B{{\mathbb B}}
\def\F{{\mathbb F}}
\def\K{{\mathbb K}}
\def\M{{\mathbb M}}
\def\N{{\mathbb N}}
\def\bbO{{\mathbb O}}
\def\bP{{\mathbb P}}
\def\re{{\mathbb R}}
\def\T{{\mathbb T}}
\def\U{{\mathbb U}}
\def\na{{\mathbb N}}

\def\cA{{\mathbb A}}
\def\cB{{\mathcal B}}
\def\cC{{\mathcal C}}
\def\cD{{\mathcal D}}
\def\cF{{\mathcal F}}
\def\cG{{\mathcal G}}
\def\cK{{\mathcal K}}
\def\cM{{\mathcal M}}
\def\cO{{\mathcal O}}
\def\cQ{{\mathcal Q}}
\def\cR{{\mathcal R}}
\def\cS{{\mathcal S}}
\def\cU{{\mathcal U}}

\def\fB{{\mathfrak B}}
\def\fR{{\mathfrak R}}
\def\fS{{\mathfrak S}}

\def\a{\alpha}
\def\be{\beta}
\def\ga{\gamma}
\def\de{\delta}
\def\De{\Delta}
\def\e{\varepsilon}
\def\la{\lambda}
\def\vr{\varphi}
\def\si{\sigma}
\def\Si{\Sigma}
\def\osi{{\overline{\sigma}}}
\def\om{\omega}

\def\lv{\left\vert}
\def\rv{\right\vert}
\def\lV{\left\Vert}
\def\rV{\right\Vert}
\def\lb{{\right]}}
\def\rb{{\left[}}
\def\lB{\left[}
\def\rB{\right]}

\def\then{\Longrightarrow}
\def\ov{\overline}
\def\longto{\longrightarrow}

\def\oom{{\overline{\omega}}}
\def\KS{{\K\times\Si}}
\def\oU{{\overline{U}}}
\def\ow{{\overline{w}}}
\def\ox{{\overline{x}}}
\def\CK{{C^\a(\K,\re)}}
\def\CS{{C^\a(\Si,\re)}}

\title{Ergodic Transport Theory, periodic maximizing probabilities  and  the twist condition}

\author{G. Contreras (*), A. O. Lopes \,(**) and  E. R. Oliveira (***)}

\date{\today \,\,-\,(*) CIMAT, Guanajuato, CP: 36240, Mexico \,(**) \,Instituto de Matem\'atica, UFRGS, 91509-900 - Porto Alegre, Brasil. Partially supported by CNPq, PRONEX -- Sistemas
Din\^amicos, INCT -- IMPA, and beneficiary of CAPES financial support,\, (***) \,Instituto de Matem\'atica, UFRGS, 91509-900 - Porto Alegre, Brasil}

\maketitle

\begin{abstract}

Consider the shift $T$ acting on the Bernoulli space $\Sigma=\{1, 2,
3,.., d\}^\mathbb{N} $ and $A:\Sigma \to \mathbb{R}$ a Holder
potential.
Denote
$$m(A)\,=\,\max_{\nu \text{\, an invariant probability
for $T$}} \int A(x) \; d\nu(x),$$ and, $\mu_{\infty,A},$ any
probability which attains the maximum value. We will assume that the maximizing probability $\mu_\infty$ is unique and has support in a periodic orbit. We denote by $\T$ the
left-shift acting on the space of points $(w,x) \in \{1, 2, 3,..,
d\}^\mathbb{Z} =\Sigma\times \Sigma=\hat{\Sigma}$. For a given
potential Holder  $A:\Sigma \to \mathbb{R}$, where $A$ acts on the variable $x$, we say that a Holder
continuous function $ W: \hat{\Sigma} \to \mathbb{R}$ is a
involution kernel for $A$ (where $A^*$ acts on the variable $w$), if there is a Holder function $A^*:\Sigma
\to \mathbb{R}$, such that,
$$A^*(w)= A\circ \T^{-1}(w,x)+ W \circ \T^{-1}(w,x) -
W(w,x).$$

Such $A^*$ and $W$ exists when $A$ is Holder.
One can also consider $V^*$ the calibrated subaction for $A^*$, and, the maximizing probability $\mu_{\infty,A^*}$ for $A^*$.

The following result expression is known:  for any given $x\in\Sigma$, it is true the relation
$$V(x) =
\sup_{ w \in \Sigma}\, [\,(W(w,x) - I^*(w))\, -\, V^*
(w)\,],
$$
where $I^*$ is non-negative lower semicontinuous function (it can attain the value $\infty$ in some points). In this way $V$ and $V^*$ form a dual pair. $I^*$ is a deviation function for a Large Deviation Principle.

For each $x$ one can  get one (or, more than one) $w_x$ such attains the supremum above. That is, solutions of
$$V(x) =
W(w_x,x) - V^*
(w_x)- I^*(w_x)\,.$$

A pair of the form $(x,w_x)$ is called an optimal pair.

Under some technical assumptions, we show that generically on the potential $A$,
the set of possible optimal $w_x$, when $x$ covers the all range of
possible elements $x$ in $ \in \Sigma$, is finite.
\newline

\bigskip


\end{abstract}

\maketitle


\section{Introduction}

We will state in this section the mathematical definitions and concepts  we will consider in this work.
We denote by $T$ the action of the shift in the Bernoulli space
$\{1, 2, 3,.., d\}^\mathbb{N} =\Sigma$.

\begin{defi}
Denote
$$m(A)\,=\,\max_{\nu \text{\, an invariant probability
for $T$}} \int A(x) \; d\nu(x),$$ and, $\mu_{\infty,A},$ any
probability which attains the maximum value. Any one of these
probabilities $\mu_{\infty,A}$ is called a maximizing probability for
$A$.
\end{defi}

We will assume here that the maximizing probability is unique and has support in a periodic orbit. An important conjecture claims that this property is generic (see \cite{CLT} for partial results).

The analysis of this kind of problem is called Ergodic Optimization Theory \cite{Conze-Guivarch} \cite{Jenkinson1} \cite{HY}  \cite{CLT} \cite{Le}  \cite{BG} \cite{CM} \cite{TZ} \cite{GL2}. A generalization of such problems from the point of view of Ergodic Transport can be found in \cite{LM2}.

We denote the set of $\alpha$-Holder potentials on $\Sigma$ by $\CS$.

If $A$ is Holder, and, the maximizing probability for $A$ is unique, then the probability $\mu_{\infty,A}$ is the limit of the Gibbs
states $\mu_{\beta\, A}$, for the potentials $\beta  A,$ when $\beta \to \infty$ \cite{CLT} \cite{CG}. Therefore, our analysis concerns Gibbs probabilities at zero temperature $(\beta=1/T$).

The metric $d$ in $\Sigma$ is defined by
 $$
 d(\om,\nu):= \la^N,\qquad N:=\min\{\,k\in\na\,|\,\om_k\ne \nu_k \,\}.
 $$

The norm we consider in the set $\CS$ of $\alpha$-Holder potentials $A$ is
$$
||A||_\alpha=\sup_{d(x,y)<\e}\frac{\lv A(x)-A(y)\rv}{d(x,y)^\a}+ \sup_{x \in \Sigma}\, |A(x)|. $$

 We denote $\hat{\Sigma}= \Sigma \times \Sigma= \{1, 2, 3,.., d\}^\mathbb{Z}$, and, we use the notation $\bar x=(...x_{-2},x_{-1}\,|\,x_0, x_1, x_2,..)\in \hat{\Sigma},$ \,$w=(x_0, x_1, x_2,..)\in \Sigma$, $x=(x_{-1},x_{-2},..) \in \Sigma$. Sometimes we denote $\bar x=(w,x).$  We say $w=(x_0, x_1, x_2,..)$ are the future coordinates of $\bar x$ and $x=(x_{-1},x_{-2},..) \in \Sigma$ are the past coordinates of $\bar x$. We also use the notation: $\sigma$ is the shift acting in the past coordinates $w$ and $T$ is the shift acting in the future coordinates $x$. Moreover, $\T $ is the (right side)-shift on $\hat{\Sigma}$, that is,
$$ \T^{-1}  (...x_{-2},x_{-1}\,|\,x_0, x_1, x_2,..)= (...x_{-2},x_{-1}\,x_0\, |\, x_1, x_2,..).$$

As we said before we denote the action of the shift in the coordinates $x$ by $T$, that is,
$$T(x_{-1},x_{-2},x_{-2}..)= (x_{-2},x_{-3},x_{-4}..).$$

For $w=(x_0, x_1, x_2,..)\in \Sigma$, $x=(x_{-1},x_{-2},..) \in \Sigma$, denote $\tau_w\,(x)= (x_0\,,x_{-1},x_{-2},..).$

We use sometimes the simplified notation $\tau_w=\tau_{x_0}$, because the dependence on $w$ is just on the first symbol $x_0$,


Using the simplified notation $\bar x=(w,x)$, we have
$$ \T^{-1} (w,x) = ( \,\sigma(w), \,\tau_{w}(x)  \,).$$

We also define $\tau_{k,a} x= (a_{k},a_{k-1}, ...,a_1, x_0,x_1,x_2,...) $, where $x=(x_0,x_1,x_2,...),$ $a=(a_1,a_2,a_3,...)$.

Bellow we consider that $A$ acts on the past coordinate $x$ and  $A^*$ acts on the future coordinate $w$.
\begin{defi}

Consider $A:\Sigma \to \mathbb{R}$ Holder.
We say that a Holder continuous function $ W: \hat{\Sigma} \to \mathbb{R}$ is a involution kernel for $A$, if
there is a Holder function  $A^*:\Sigma \to \mathbb{R}$, such
that,
$$A^*(w)= A\circ \T^{-1}(w,x)+ W \circ \T^{-1}(w,x) -
W(w,x).$$

We say that $A^* $ is a dual potential of $A$, or, that $A$ and $A^*$ are in involution.
\end{defi}

The above expression can be also written as
$$A^*(w)= A(\tau_{w}(x) ) + W ( \,\sigma(w), \,\tau_{w}(x)  \,) -
W(w,x).$$

\begin{theorem}
\cite{BLT} Given $A$ Holder there exist a Holder $W$ which is an involution kernel for $A$.
\end{theorem}

We call $\mu_{\infty,A^*}$ the maximizing probability for $A^*$ (it is unique if $A$ has a unique maximizing $\mu_{\infty \, A}$, as one can see in \cite{BLT})
\bigskip

To consider a dual problem is quite natural in our setting. Note that $m(A)=m(A^*)$ (see \cite{BLT}).

We denote by $\mathbb{ K}=\mathbb{ K}( \mu_{\infty,A}, \mu_{\infty,
A^*})$ the set of probabilities $\hat{\eta} (w,x)$ on
$\hat{\Sigma}$, such that $ \pi_ x^* (\hat{\eta} ) =
\mu_{\infty,A}, \, \, $ and also that $\,\,\pi_ y^* (\hat{\eta} ) =
\mu_{\infty,A^*} \,.$

We are interested in the solution $\hat{\mu  }$  of  the Kantorovich
Transport Problem for $-W$, that is, the solution of
$$
\,\,  \inf_{\hat{\eta} \in \mathbb{ K}  } \int \int
- W(w,x)  \, d\,\hat{\eta}. \,
$$

Note that in the definition of $W$ we use the dynamics of $\T$. Note also that if we consider a new cost of the form $c(x,w)=-W(x,w) + \varphi (w)$, instead of $-W$, where $\varphi$ bounded and measurable, then we {\bf do not} change the original minimization  problem.

\begin{defi}
A calibrated sub-action $V:\Sigma\to \mathbb{R}$ for the potential
$A$, is a continuous function $V$ such that

$$ \sup_{y\, \text{such that}\, T(y)\,=\,x}\,  \,\{\,V(y) \,+ \,A(y) \,-\, m ( A)\,\}\,= \, V( x)  . $$
\end{defi}

We denote by $V^*$ a calibrated subction for $A^*$.

If $V$ is calibrated it is true that for any $z$ in the support of the maximizing probability $V(z)+A(z)-m(A)=V(T(z).$ Generically (on the Holder class) on $A$, one can show that  this equality it true just on the support of the maximizing probability.
An important issue here
is to know  if  this last property is also true for the maximizing probability of $A^*$. We address this question in the last sections.

We denote by $\hat{\mu  }$  the minimizing probability over  $\hat{\Sigma}=\{1, 2, 3,.., d\}^\mathbb{Z}$ for the natural Kantorovich
Transport Problem associated to the $-W$, where $W$ is the involution kernel for $A$ (see \cite{BLT}).

We call  $\hat{\mu  }_{max}$  the natural extension of $ \mu_{\infty,A}$ as described in \cite{BLT} \cite{PY}.

In \cite{LOT} was shown (not assuming the maximizing probability is a periodic orbit) that:

\begin{theorem} Suppose the maximizing probability for $A$ Holder is unique (not necessarily a periodic orbit). Then, the minimizing Kantorovich
probability $\hat{\mu}$ on $\hat{\Sigma}$  associated to $-W$, where $W$ is the involution kernel for $A$, is
$\hat{\mu  }_{max}$.
\end{theorem}

Moreover, it was shown in \cite{LOT}:

\begin{theorem} \label{prec} Suppose the maximizing probability is unique (not necessarily a periodic orbit). If $V$ is the calibrated subaction for $A$, and  $V^*$ is the calibrated  subaction for $A^*$, then,
the pair $(-V,-V^*)$ is the dual ($-W$)-Kantorovich  pair of $(\mu_{\infty,A},\mu_{\infty,A^*})$.
\end{theorem}

The solution cames from the so called complementary slackness condition \cite{Boyd} \cite{Vi1} \cite{Vi2} which were obtained in proposition 10 (1) \cite{BLT}.
We can assume that $\gamma$ in \cite{BLT} is equal to zero.

One can consider in the Bernoulli space
$\Sigma=\{0,1\}^\mathbb{N}$ the lexicographic order. In this way,
$x<z$, if and only if, the first element $i$ such that, $x_j=z_j$
for all $j<i$, and $x_i \neq z_i$, satisfies the property $x_i<
z_i$. Moreover, $(0,x_1,x_2,...)<   (1,x_1,x_2,...).$

One can also consider the more general case
$\Sigma=\{0,1,...,d-1\}^\mathbb{N}$, but in order to simplify the
notation  and to avoid technicalities, we consider only the case
$\Sigma=\{0,1\}^\mathbb{N}$. We also suppose, from now on, that
$\hat{\Sigma}= \Sigma\times \Sigma.$


\begin{defi} We say a  continuous  $G: \hat{\Sigma}=\Sigma\times \Sigma \to \mathbb{R}$ satisfies the twist condition on
$ \hat{\Sigma}$, if  for any $(a,b)\in \hat{\Sigma}=\Sigma\times
\Sigma $ and $(a',b')\in\Sigma\times \Sigma $, with $a'> a$,
$b'>b$, we have
\begin{equation}
G(a,b) + G(a',b')  <  G(a,b') + G(a',b).
\end{equation}
\end{defi}

\begin{defi} We say a continuous $A: \Sigma \to \mathbb{R}$   satisfies the twist condition, if its  involution kernel $W$
satisfies the twist condition.
\end{defi}

Examples of twist potentials are presented in \cite{LOT}.

One of the main results in \cite{LOT} is:

\begin{theorem} Suppose the maximizing probability is unique (not necessarily a periodic orbit) and $W$ satisfies the twist condition on $\hat{\Sigma}$, then, the support of $\hat{\mu}_{max}= \hat{\mu}$ on $\hat{\Sigma}$ is a graph (up to one possible orbit).

\end{theorem}

Given $A:\Sigma=\{1,2,..,d\}^\mathbb{N} \to \mathbb{R}$ Holder, the  Ruelle operator
$\mathcal{L}_A : C^{0} (\Sigma) \to  C^{0} (\Sigma)$ is given by $\mathcal{ L}_A (\phi) (x)= \psi(x) =
\sum_{T (z)=x} \, e^{A(z)} \, \phi(z).$

We also consider   $\mathcal{ L}_{A^*} $ acting on continuous functions in $\Sigma=\{1,2,..,d\}^\mathbb{N}$.

We also denote by $\phi_A$ and $\phi_{A^*}$ the corresponding
eigen-functions associated to the main common eigenvalue $\lambda(A) $ of the operators $\mathcal{ L}_{A} $ and $\mathcal{ L}_{A^*} $ (see \cite{PP}).

$\nu_A$, and, $\nu_{A^*}$ are respectively the eigen-probabilities
for the dual of the Ruelle operator $\mathcal{ L}_{A} $ and $\mathcal{ L}_{A^*} $.
\vspace{0.2cm}

Finally, $\mu_A   = \nu_{A}  \, \phi_{A} =  $ and $  \mu_{A^*} = \nu_{A^*} \, \phi_{A^*}  $
are the invariant probabilities such that  they are solution of the
respective pressure problems for $A$ and $A^*$. The probability $\mu_A$ is called  the Gibbs measure for the potential A.

If the maximizing probability is unique, then, it is easy to see that considering for any real $\beta$ the potential $\beta A$ and the corresponding $\mu_{\beta\, A }$, then,
$\mu_{\beta\, A } \to \mu_{\infty \,, A} $, when $\beta \to \infty$.

In the same way, if we take any real $\beta$, the potential $\beta A^*$, and, the corresponding
 $\mu_{\beta\, A^* }$, then $\mu_{\beta\, A^* }\to \mu_{\infty \,, A^* }$, when $\beta\to \infty$.

 In Statistical Mechanics $\beta=\frac{1}{T}$ where $T$ is temperature. Then, $\mu_{\infty,A}$ is the version of the
 Gibbs probability at temperature zero.
 \vspace{0.3cm}

One can choose $c$ (a normalization constant) such that

\begin{equation} \int\, \int \, e^{W(w,x)-c} \, d \nu_{A^*} (w) \,d \nu_{A} (x)\,=\,1,
\end{equation}

A calibrated  sub-action  $V$ can obtained as the limit \cite{CLT}
\cite{CG}
\begin{equation}
 V(x) = \lim_{\beta \to \infty} \frac{1}{\beta} \, \log
\phi_{\beta A} (x).
\end{equation}

In the same way  we can get a calibrated
sub-action $V^*$ for $A^*$ by taking
\begin{equation}
V^*(w) = \lim_{\beta \to \infty} \frac{1}{\beta} \, \log
\phi_{\beta A^*} (w)\,\, . \end{equation}

Moreover, by \cite{BLT}
\begin{equation}
\phi_{ A^*} (w) = \int e^{\, W_A(w,x) - c}
\, d \nu_{ A} (x),
\end{equation}
and,
\begin{equation} \label{eq1} \phi_{ A} (x) = \int e^{\, W_A(w,x) - c}
\,d \nu_{ A^*} (w)=  \int e^{\, W_A(w,x) - c}
\,\frac{1}{\phi_{ A^*}} d \mu_{ A^*} (w).
\end{equation}

\vskip 0.5cm

Note that if $W$ and $A^*$ define an involution for $A$, then, given any real $\beta$,
we have that $\beta\, W$ and $\beta\,A^*$ define an involution for the potential  $\beta\, A$. The normalizing constant $c$, of course, changes with $\beta$ (see \cite{BLT}).

\bigskip

  We say a family of probabilities $\mu_\beta$, $\beta \to \infty$,  satisfies a Large Deviation Principle, if there is function $I:\Sigma \to \mathbb{R}\cup \infty$, which is non-negative, lower semi-continuous, and such that, for any cylinder $K\subset \Sigma$, we have
\[\lim_{\beta\to\infty} \frac{1}{\beta}\log(\mu_{\beta}(K))=-\inf_{z\in K}I(z).\]

In this case we say the $I$ is the deviation function. The function $I$ can take the value $\infty$ in some points.

In \cite{BLT} a Large Deviation Principle is described for the family $\mu_\beta=\mu_{\beta A}$ of equilibrium states for $\beta A$, under the assumption that the maximizing probability for $A$ is unique (see also \cite{LM} for a different case where it is not assumed uniqueness). The function $I$ is zero on the support of the maximizing probability for $A$. We point out that there are examples where $I$ can be zero outside the support (even when the maximizing probability is unique) as we will show bellow in an example due to R. Leplaideur.

Applying the same to $A^*$ we get a deviation function $I^*:\Sigma\to \mathbb{R}\cup\{\infty\}$.
\bigskip

The function $I^*$ is defined by
$$I^*(w)=\sum_{n\geq\, 0}\big( V^*\circ\sigma-V^*-A^* \big)\circ
\sigma^n(w),$$
where $V^*$ is a fixed calibrated subaction.
\bigskip

{\it We point out that in fact the claim of theorem \ref{prec} should say more precisely that
the pair $(-V,-V^*)$ is the dual ($-W+I^*$)-Kantorovich  pair of $(\mu_{\infty,A},\mu_{\infty,A^*})$.}
\bigskip

Using the property (\ref{eq1}) above, and, adapting  Varadhan's Theorem \cite{DZ} to the present setting, it is shown in \cite{LOT}, that given any $x\in\Sigma$, then, it is true the relation
$$V(x) =
\sup_{ w \in \Sigma}\, [\,W(w,x) - V^*
(w)- I^*(w)\,].
$$

For each $x$ we get one (or, more than one) $w_x$ such attains the supremum above. Therefore,
$$V(x) =
W(w_x,x) - V^*
(w_x)- I^*(w_x)\,.$$

{\bf A pair of the form $(x,w_x)$ is called an optimal pair.} We can
also say that $w_x$ is optimal for $x$.
\bigskip

{\it Given $A$ the involution kernel $W$ and the dual potential $A^*$ are not unique. But, the above maximization  problem is intrinsic on $A$. That is, if we take another $A^*$ and the corresponding $W$, there is some canceling, and we get the same problem as above (given $x$ the optimal $w_x$ does not change)}.
\bigskip

It is also true that (see \cite{LOT}) there exists $\gamma$ such that for any $w$
$$\gamma+V^*(w) =
\sup_{ x \in \Sigma}\, [\,W(w,x) - V
(x)- I(x)\,].
$$

Given $w$, a solution of the above maximization is denoted by $x_w$. {\bf We denote the pair $(x_w,w)$ a $*$-optimal pair.}

We assume without lost of generality that $\gamma=0$, by adding $-\gamma$ to the function $W$.

\begin{defi} The set of all $(x, w_x)$, is called the optimal set for $A$, and, denoted by $\mathbb{O}(A)$.

\end{defi}

{\bf Remark 1:} Note that in \cite{LOT} it was shown that, under the twist assumption, the support of the optimal transport periodic probability $\hat{\mu}_A$, for the cost $-W$, is a graph, that is, for each $x$, there is only one $w$ such that $(x,w)$ is in the support of $\hat{\mu}_A$.
But, nothing is said about the  graph property of the set  $\mathbb{O}(A)$. Note that the support of $\hat{\mu}_A$ is contained in $\mathbb{O}(A)$.

{\bf Remark 2:} Note also that the minimal transport problem for the cost $-W$, or the cost $-W + I^*$ is the same (see \cite{LOT}).
\bigskip

Our main result is:

\begin{theorem} Generically, in the set of potentials Holder potentials $A$ that satisfy

(i) the twist condition,

(ii) uniqueness of maximizing probability which is supported in a periodic orbit,

 the set of possible optimal $w_x$, when $x$ covers the all range of possible elements $x$ in $ \in \Sigma$, is finite.

\end{theorem}

We point out that this is a result for points outside the support of the maximizing probability.

\vspace{0.5cm}

In the first two sections we will show that under certain conditions the set of possible optimal $w_x$ is finite, for any $x$.
In section 3 and 4 we will show that these conditions are generic.
\bigskip

 In \cite{LOS}
it is also considered the twist condition and results for the
analytic setting are obtained.

\bigskip

\section{The twist property}

In order to simplify the notation we assume that
$m(A)=m(A^*)=0$.
\bigskip

Given $A$, we denote
$$\Delta(x,x',y)=\sum_{n\geq1}A\circ\tau_{y,n}(x)
-A\circ\tau_{y,n}(x').$$

The involution kernel $W$ can be computed for any $(\om,x)$ by
$W(\om,x)=\De_A(\om,\ox,x)$,  where we choose a point $\ox$ for good.

It is known the following relation: for any $x,x',w\in \Sigma$, we
have that $W(w,x)-W(w,x')=\Delta(w,x,x')$ (see \cite{BLT})

We assume from now on that $A$ satisfies the twist condition.
It is known in this case (see \cite{Ba} \cite{LMST}), that $x \to w_x$ (can be multivaluated) is
monotonous decreasing.

\vskip 0.5cm
\begin{pro} If $A$ is twist, then
$x \to w_x$ is
monotonous decreasing.

\end{pro}

{\it Proof:} See \cite{LOS}
$\blacksquare$\\

We define $R$ by  the expression $R (x)= V(\sigma (x) ) - V(x) - A(x)\geq 0$,  and, we define $R^*$ by  $R^* (w)= V^*(\sigma (w) ) - V^*(w) - A^*(w)\geq 0$.

Note that given $y$, there is a preimage $x$ of $y$, such that, $R(x)=0.$ The analogous property is true for $R^*$.

Given $A$ (and a certain choice of $A^*$ and $W$) the next result (which does not assume the twist condition) claims that the dual of $R$ is $R^*$, and the corresponding involution kernel is
$(V^* + V - W).$

\begin{pro} (Fundamental Relation)(FR)

$$ R(\tau_{w}x)= (V^* + V - W)(x,w) - (V^* + V - W)( \tau_{w}x , \sigma(w)) + R^* (w).$$

\end{pro}

\textbf{Proof:} see \cite{LOS}
$\blacksquare$

\vskip 0.5cm

We know that the calibrated subaction satisfies
$$V(x)= \max_{w \in \Sigma} (-V^* -I^* +  W)(x,w).$$
Then, we define
$$b(x,w)= (V^* + V  +I^* - W)(x,w) \geq 0,$$ and,
$$\Gamma_{V}=\{(x,w) \in \Sigma \times \Sigma | V(x)=  (-V^* -I^* +  W)(x,w)\},$$
which can be written in an equivalent form
$$\Gamma_{V}=\{(x,w) \in \Sigma \times \Sigma \,| \, b(x,w)=0\}.$$

\bigskip

Given $x$, this maximum at $w_x$  can not be realized
where $I^*(w)$ is infinity.

\bigskip


{\bf Remark 3: Note, that $b(x,w)=0$, if and only if, $(x,w)$ is an optimal pair. We are not saying anything for
$*$-optimal pairs.}

If we use $ R^* (w)= I^*(w) - I^*(\sigma w)$, the FR becomes
$$R(\tau_{w}x) = (V^* + V - W)(x,w) - (V^* + V - W)(\tau_{w}x , \sigma(w)) + I^*(w) - I^*(\sigma (w)),$$
or
$$R(\tau_{w}x) = b(x,w) - b(\tau_{w}x , \sigma(w))\,\,\,\,\,\,\,\,\,\text{FR1}.$$

From this main equation we get:\\

\begin{lemma}
If $\,\T^{-1}(x,w) = (\tau_{w}x , \sigma(w))$, then

a) $b - b \circ \T^{-1} (x,w) = R(\tau_{w}x)$;

b) The function $b$ it is not decreasing in the trajectories of $\T$;

c) $\Gamma_{V}$ is backward invariant;

d) when $(x,w)$ is optimal then $R(\tau_w(x))=0$.
\end{lemma}

\textbf{Proof:}\\

The first one its a trivial consequence of the definition of $\T^{-1}$. The second one it is a consequence of $R \geq 0$:\\
$$b - b \circ \T^{-1} (x,w) = R(\tau_{w}x) \geq 0 $$
$$b(x,w)  \geq   b \circ \T^{-1} (x,w).$$

In order to see the third part we observe that
$$(x,w) \in \Gamma_{V} \Leftrightarrow b(x,w)=0$$
Since
$$b(x,w)  \geq   b \circ \T^{-1} (x,w) \geq 0$$
we get $b(\tau_{w}x , \sigma(w))=0$ thus  $(\tau_{w}x , \sigma(w)) \in \Gamma_{V}$. $\blacksquare$\\

 From the above we get that in the case $(x,w)$ is optimal, then, $\T^{-1} (x,w) $ is also optimal. Indeed, we have that
$$b(x,w)=0 \Rightarrow b(\tau_{w}x , \sigma(w))=0.$$

This is equivalent to
$$V(x)= - V^*(w) -I^*(w) - W(x,w) \Rightarrow $$
$$V(\tau_{w}x)= - V^*(\sigma(w)) -I^*(\sigma(w)) - W(\tau_{w}x,\sigma(w)).$$

\bigskip
{\bf In this way $\T^{-n}$ spread optimal pairs.}
\bigskip

\bigskip

We denote by $M$ the support of the maximizing probability periodic orbit.

Consider the compact set of points $P=\{w\in \Sigma$, such that $\sigma(w)\in M$, and $w$ is not on $M\}$.
\bigskip

 \begin{defi}
 We say that $A$ is good if,  for each $w\in P$, we have that $R^*(w)>0$.
\end{defi}

We alternatively, say sometimes that
 $R^*$ is good for $A^*$.

We point out that there are examples of potentials $A^*$ (with a
unique maximizing probability) where the corresponding $R^*$ is not
good (see Example \ref{ExRen} in the end of the present section) .


Remember that ,
$$I^*(w)=\sum_{n\geq\, 0}\big( V^*\circ\sigma-V^*-A^* \big)\circ
\sigma^n(w)\,=\,\sum_{n\geq\, 0}\, R^*( \sigma^n(w) )  .$$

In \cite{LMST} section 5  it is shown that if $I^*(w) $ is finite, then

$$ \lim_{n \to \infty} \frac{1}{n} \sum_{j=0}^{n-1}\delta_{\sigma^j(w)}\to \mu_\infty^* .$$

One important assumption here is  that $R^*$ is good. We will show later that this property is true for generic potentials $A$.

\begin{pro}

Assume $A$ is good. If $I^*(w)<\infty$, then, $w$ is in the preimage
of the maximizing probability for $A^*$.

\end{pro}

For a proof see section 5 in \cite{LMST}

\begin{pro}

Suppose $(x,w_x)$ is an optimal pair, and, $A$ is good, then, there
exists $k$, such that, $\tilde{w}_k=\sigma^k(w_x)$ is in the support
of the maximizing probability for $A^*$. Moreover, for such $k$, we
have that $\T^{-k} (x,w_x)$ is an optimal pair.

\end{pro}

{\it Proof:}

Suppose  $(x,w_x)$ is optimal.
 Therefore, $I^*(w_x)<\infty.$ Then, by a previous proposition,  $w_x$ is in the pre-image of the maximizing probability for $A^*$, that is, there exists $k$ such that $\tilde{w}=\sigma^k(w_x)$ is in the support of the maximizing probability for $A^*$.

 Moreover, $b(\T^{-k} (x, w_x))=(x_k , \sigma^k (w_x))=(x_k, \tilde{w}_k)$ is also optimal.  $\blacksquare$\\

\begin{pro}

Assume $A$ is good, then, the set of $w$ such that $I^*(w)<\infty$
is countable.

\end{pro}

\begin{defi} We say a continuous $A: \Sigma \to \mathbb{R}$   satisfies the
the countable  condition, if there are a countable number of
possible optimal $w_x$, when $x$ ranges over the interval $\Sigma$.

\end{defi}

{\bf Remark:} If $A$ is good, then, it satisfies the countable
condition.
\bigskip

We showed before that  the twist property implies that for $x<x'$, if $b(x,w)=0$ and $b(x',w')=0$, then $w' < w$, which means that the optimal sequences are monotonous  not increasing.
Thus,  we define the \textbf{``turning point $c$"}  as being the maximum of the point $x$ that has his optimal sequence starting in 1:
$$c=\sup\{x\, |\, b(x,w)=0 \,\Rightarrow \, w=(1\, w_1\, w_2 ...)\}.$$


The main criteria is the following:\\

\emph{``If $x \in \sigma$ has the optimal sequence  $w=(w_0 \, w_1 \, w_2\, ...)$ then
\[
w_0= \left \{
  \begin{array}{ll}
    1, \,\, if &  x \in [0^\infty,c] \\
    0, \,\, if &  x \in (c,1^\infty]"
  \end{array}
\right.
\]}
Starting from $(x_0,w_0)$ we can iterate FR1 by $\T^{-n} (x,w) =(x_n,w_n)$ in order to obtain new points $w_1 , \, w_2\, ...\in \Sigma$.
Unless the only possible optimal point $w_x$, for all $x$, is a fixed point for $\sigma$, then,  $0<c<1$.

Note that for $c$ there are two optimal pairs $(c,w)$ and $(c,w')$, where the first symbol of $w$ is zero, and, the first symbol of $w'$ is one.



We denote
$$B(w)=\{ x \, | \, b(x,w)=0 \}.$$

\begin{lemma} (Characterization of optimal change)
Let $c \in (0^\infty,1^\infty)$ be the turning point then, for any
$x < x'$ and $b(x,w)=0$ and $b(x',w')=0$, we have $w \neq w'$  if,
and only if, there exists $n \geq 0$ such that $T^n (c) \in [x,x']$.
Moreover, if $x,x'$ are such that $w_x$ and $w_x'$ are identical
until the $n$ coordinate, then, $T^n (c)\in (x,x').$

Each set $B(w)=[a,b]$ is such that $a=T^n(c)$, or, $a$ it is
accumulated by a subsequence of $T^j(c)$ from the left side. Similar
property is true for $b$ (accumulated by the right side).

\end{lemma}
\begin{proof}

Step 0\\
If $x < x' \leq c$ then $w_{0} = w'_{0}=1$ else if $c< x < x' $ then $w_{0} = w'_{0}=0$.
Suppose $w_{0} = w'_{0}= \in \{0,1\}$ then applying FR1 we get  $\tau_{i}x < \tau_{i}x'$ and $b(\tau_{i}x , (w_{1} \, w_{2} \, ...))=0$ and $b(\tau_{i}x' , (w'_{1} \, w'_{2} \, ...))=0$.

Step 1\\
If $\tau_{i}x < \tau_{i}x' \leq c$ then $w_{1} = w'_{1}=1$ else if $c< \tau_{i}x < \tau_{i}x' $ then $w_{1} = w'_{1}=0$.  Otherwise if $\tau_{1}x <  c < \tau_{1}x' $ we can use the monotonicity of $T$ in each branch in order to get $ x <  T(c) < x' $. Thus $$w_{1} \neq w'_{1} \Leftrightarrow  x <  T(c) < x'. $$
The conclusion comes by iterating this algorithm.

The last claim is obvious from the above.

\end{proof}
\bigskip

A point $x$ is called pre-periodic (or, eventually periodic) if
there is $n\neq m $, such that, $T^n(x)=T^m(x).$
\bigskip

We denote $[a,b]$, $a,b\in \Sigma$, $a \leq b$, the set of all
points $w$ in $\Sigma$ such $a\leq w\leq b$. We call $[a,b]$ the
interval determined by $a$ and $b$. Each interval $[a,b]$, with
$a<b$, is not countable

\begin{lemma}

The set
$$B(w)=\{ x \, | \, b(x,w)=0 \}$$
is an interval (can eventually  be a single point). More
specifically, if $B(w)=[a,b]$, then, $a$ and $b$ are adherence
points of the orbit of $c$.

In particular, if $c$ is pre-periodic, then, for any non-empty
$B(w)$, there exists $n,m$ such that $B(w)=[T^n (c), T^m (c)]$
(unless $B(w)$ is of the form $[0,b]$, or $[a,1]$.
\end{lemma}

For a proof see \cite{LOS}.

\begin{lemma}\label{iso}
Let $c \in \Sigma$ be the turning point.  Let us suppose the $c$ is
isolated from his orbit, which means that, there is $d,e,w^{-}$ and
$w^{+}$, such that, $b(x,w^{-})=0$, for any $x \in (d, c]$, and,
$b(x,w^{+})=0$, for any $x \in [ c, e)$,  then, there is no
accumulation points of the orbit of $c$. In this case $c$ is
eventually periodic.
\end{lemma}

For a proof see \cite{LOS}.

{\bf Remark 4:} The main problem we have to face is the possibility
that the orbit of $c$ is dense in $[0,1]$.
\bigskip

If $c$ is eventually periodic there exist just a finite number
intervals $B(w)$ with positive length. The other $B(w)$ are reduced
to points and they are also finite.

\begin{lemma} Suppose $A$ satisfies the twist and the countable condition.
Then there is at least one $B(w)$ with positive length of the form
$(T^n (c), T^m (c))$. Moreover, for each subinterval $(a,b)$ there
exists at least one $B(w)$ with positive length of the form $(T^n
(c), T^m (c))$ contained on it. Therefore, there exists an infinite
number of such intervals.

\end{lemma}

Denote the possible $w$, such that, $I^*(w)< \infty$, by $w_j$, $j
\in \mathbb{N}$.

For each $w^j$, $j \in \mathbb{N}$, denote $I_j=B(w^j)$, the maximal
interval where for all $x \in I_j$, we have that, $(x,w^j)$ is an
optimal pair. Some of these intervals could be eventually a point,
but, an infinite number of them have positive length, because the
set $\Sigma$ is not countable. We consider from now on just the ones
with positive length.

Note that by the same reason, in each subinterval $(e,u)$, there
exists an infinite countable number of $B(w)$ with positive length.

We suppose, by contradiction, that each interval $B(w)=[a,b]$, with
positive length is such that, each side is approximated by a
sub-sequence of points $T^j(c)$.

Take one interval $(a_1,b_1)$ with positive length inside $\Sigma$.
There is another one $(a_2,b_2)$  inside  $(0^\infty,a_1)$, and one
more $(a_3,b_3)$ inside $(b_1,1^\infty)$.

If we remove from the interval $\Sigma$ these three intervals we get
four intervals. Using our hypothesis, we can find new intervals with
positive length inside each one of them. Then we do the same removal
procedure as before. This procedure is similar to the construction
of the Cantor set. If we proceed inductively on this way, the set of
points $x$ which remains after infinite steps is not countable. An
uncountable number of such $x$ has a different $w_x$. This is not
possible because the optimal $w_x$ are countable.

Then, the first claim of the lemma is true.

Given an interval $(a,b) \subset \Sigma$, we can do the same and use
the fact that $(a,b)$ is not countable.

$\blacksquare$\\

\begin{lemma}
Under the twist and the countable condition the turning point $c$ is
eventually periodic.

\end{lemma}

{\it Proof:}

Denote by $w^j$, $j \in \mathbb{N}$, the countable set of pre-images
of the periodic orbit maximizing $A^*$.

For each $w^j$, $j \in \mathbb{N}$, denote $I_j=B(w^j)$, the maximal
interval where for all $x \in I_j$, we have that, $(x,w^j)$ is an
optimal pair. Some of these intervals could be, eventually, a point,
but, an infinite number of them have positive length, because the
set $\Sigma$ is not countable

We will suppose $c$ is not eventually periodic, and, we will reach a
contradiction. Therefore, if $T^n(c) = T^m (c),$ then $m=n$.

From, now on we consider just $j$, such that, for the corresponding
$w_j$, the interval $I_j$ has positive length, and, it is of the
form $[T^n(c), T^m(c)]$, $m,n\geq 1$. From last lemma there exist an
infinite number of them.

Denote by $I_j=[a_j,b_j]$. We denote $I_0$ the interval of the form $[0^\infty,b_0]$, and,  $I_1$ the interval of the form $[a_0,1^\infty]$. From last lemma, for $j\neq 0,1$, there is
$n_j$ and $m_j$, such that $a_j= T^{n_j}(c)$ and
$b_j= T^{m_j}(c)$.

Consider the inverse branch $\tau_{i_1^j}$, where $i_1^j=i_1$ is
such that $\tau_{i_1}((T^{n_j})(c) )= T^{n_j-1}(c)$. This $i_1$ {\bf
do not have to be the first symbol of the optimal $w$  for
$T^{n_j}(c)$}. Then, $\tau_{i_1}(I_j)$ is another interval, which is
strictly inside a domain of injectivity of $T$, does not contain any
forward image of $c$, and in its left side we have the point
$T^{n_j-1}(c)$.

 Then, repeating the same procedure
inductively, we get $i_2$, such that $\tau_{i_2}((T^{n_j-1})(c) )=
T^{n_j-2}(c)$, determining another interval which does not contain
any forward image of $c$, and in his left side we have the point
$T^{n_j-2}(c)$. Repeating the reasoning over and over again, always
taking the same inverse branch which contain $T^n(c)$, $0\leq n \leq
n_j$, after $n_j$ times we arrive in an interval of the form $(c,
r_j)$. Note that each inverse branch preserves order. It is not
possible to have an iterate $T^k(c)$, $k \in \mathbb{N}$, inside
this interval $(c, r_j)$ (by the definition of $I_j$). Then, the
optimal $w$ for $x$ in this interval $(c, r_j)$ is a certain
$\tilde{w_j}$ which can be different of $\sigma^{n_j} (w_j).$

Suppose now that  $n_j>m_j$. Using the analogous procedure we get
that there exists $r^j$, such that the optimal $w_x$ for $x$ in the
interval $(r^j, c)$ is $\sigma^{n_j} (w_j).$

If both cases happen, then $c$ is eventually periodic.

The trouble happens when just one type of inequality is true.
Suppose without lost of generality that we have always $n_j<m_j$,
for all possible $j$.

Let's fix for good a certain $j$.

Therefore, all we can get with the above procedure is that  $c$ is
isolated by the right side

In the procedure of taking pre-image of $T^{n_j}(c)$, always
following the forward orbit $T^n(x)$, $0\leq n \leq n_j$, we will
get a sequence of $i_1,i_2,...,i_{n_j}$.  In the first step  we have
two possibilities: $\tau_{i_1} (T^{m-j}(c)) = T^{m_j-1}(c)$, or not.

If it happens the second case, we are done.  Indeed, the interval
$\tau_{i_1} [T^{n_j}(c),T^{m_j}(c))]$ does not contain forward
images of $c$ (otherwise $[T^{n_j}(c)),T^{m_j}(c))]$ would also
have). Now we follow the same procedure as before, but, this time
following the branches which contains the orbit of $T^m (c)$, $0\leq
m\leq m_j$. In this way, we get that $c$ is isolated by the left
side.

Suppose  $\tau_{i_1} (T^{m_j}(c))  = T^{m_j-1}(c)$. Consider the
interval, $[T^{n_j-1}(c)),T^{m_j-1}(c))]$, which do not contain
forward images of $c$.

Now, you can ask the same question:  $\tau_{i_2} (T^{m_j-1}(c)) =
T^{m_j-2}(c)$? If this do not happen (called the second option),
then, in the same way as before, we are done ($c$ is also isolated
by the right side). If the expression is true, then, we proceed with
the same reasoning as before.

We proceed in an inductive way until time $n_j$. If in some time we
have the second option, we are done, otherwise, we show that any
$x\in (c, T^{m_j-n_j}(c))$ has a unique optimal $w_x$ (there is no
forward image of $c$ inside it).

Denote $k=m_j-n_j$ for the $j$ we fixed.

From the above we have that for any $B(w)$, which is an interval of
the form $[T^{n_i}(c)),T^{m_i}(c))]$, for any possible $i$, it is
true that $m_i-n_i=k.$ There are an infinite number of intervals of
this form.

We claim that the set of points $x$ which are extreme  points of any
$B(\tilde{w})$,  and, such that $x$ can be approximated by the
forward orbit of $c$ is finite. Suppose without lost of generality
that $x$ is the right point of a $B(w)=(z,x)$.

If the above happens, then, by the last lemma, we have an infinite
sequence of intervals of the form $[T^{n_i}(c)),T^{n_i+k}(c))]$,
such that $T^{n_i}(c)\to x$, as $n_i\to \infty.$ Therefore, $x$ is a
periodic point of period $k$. There are a finite number of points of
period $k$. This shows our main claim. Finally, $c$ is eventually
periodic.

$\blacksquare$\\

\begin{defi}
We denote by $\mathbb{G}$ the  set of Holder potentials such that

 1) the maximizing probability is unique, and it is a periodic orbit;

 2) The potential $A$ satisfies the twist condition;

 3) $R^*$ is good for $A^*$.
\end{defi}

From the above one can get:

\begin{theorem} Suppose $A$ is in $\mathbb{G}$,
then, for any $x$, there exists a finite number of possible $w_x$ such that $(x,w_x)$
are optimal pairs. If we denote by $I_j=(a_j,b_j)$, $j=1,2,..,n$,
the maximal open intervals where for $x\in I_j$ the $w_x$ is
constant, then, just on the points $x=a_j$, or $x=b_j$, we can get
two different $w_x$, which define points $(x,w_x)$ in the optimal
set $\mathbb{O}(A)$.

 \end{theorem}

 Note that if $x$ is in the maximizing orbit for $A\in \mathbb{G}$, then,  at
 least one of the optimal $w_x$ is in the support of the maximizing probability for $A^*$. This point $x$ can be eventually in the extreme of one of this maximal intervals $I_J$. This do not contradicts the graph property

\begin{defi}
We denote by $\mathbb{A}$ the  open set of Holder potentials such that

 1) the maximizing probability is unique, and it is a periodic orbit;

 2) The potential $A$ satisfies the twist condition.
\end{defi}

 The next theorem shows the class we consider above is large.

\begin{theorem} The set $\mathbb{G}$ is generic in the open set $\mathbb{A}$.

 \end{theorem}

The proof of this result will be done in the next two sections (see Theorem \ref{MTP} bellow).

\begin{corollary} For any $A \in \mathbb{ G}$, the value $c=c_A$ is given by the expression
$$ c=\inf\{x\,|\, V(x) - V( \tau_1 x) - A (\tau_1 (x))>0\},$$
where $V$ is any calibrated subaction for $A$. Moreover,
$c_A$ is locally constant as a function of $A$.

\end{corollary}

{\it Proof:}

The first claim follows from the fact that we have to use $0$ as first symbol of the optimal $w_x$, when $x$ is on the right of $c$.

If $R^*$ is good the point $c$ is in the pre-image of the support of the maximizing probability (which is locally constant by the continuous varying support property \cite{CLT}). Therefore, the possible $c$ are in a countable set.

Note that under the uniqueness hypothesis of the maximizing probability for $A$, the sub-action $V=V_A$ can be chosen in a continuous fashion with $A$. From this follows the last claim of the corollary. $\blacksquare$\\

Now we will provide a counterexample.

\begin{example}\label{ExRen}

The following example is due to R. Leplaideur.

We will show an example on the shift where the maximizing probability for a certain Lipschitz potential $A^* :\{0,1\}^\mathbb{N} \to \mathbb{R}$ is a unique periodic orbit $\gamma$  of period two, denoted by $ p_0=(01010101...),p_1=(10101010...)$, but for a certain point, namely,  $w_0=(110101010..)$, which satisfies $\sigma(w_0)= p_1$,
we have that $R^*(w_0)=0$.

\bigskip

The potential $A^*$ is given by $A^*(w)=-d(w, \gamma\cup \Gamma)$, where $d$ is the usual distance in the Bernoulli space. The set $\Gamma$ is described later.



\bigskip
For each integer $n$, we  define a  $2n+3$-periodic orbit $z_n,\sigma(z_{n}),\ldots,\sigma^{2n+2}(z_{n})$ as follows:

we first set
$$b_n=  (\underbrace{01010101\ldots01}_{2\, n}1\,01),$$
and the point $z_{n}$ is the concatenation of the word $b_{n}$: $ z_n=( b_n,\,b_n,\,...)$

The main idea here is to get a sequence of periodic points which spin around the periodic orbit $\{p_0,p_1\}$ during the time $2\,n$, and then
pass close by $w_0$ (note that $d(\sigma^{2\, n} (z_n) , w_0) = 2^{- 2 (n+1)}$).

Denote $\gamma_n$ the periodic orbit $\gamma_n=\{ z_n, \sigma(z_n),\, \sigma^2(z_n),\,...,\sigma^{ 2n+2}(z_n)\,\}.$

Consider the sequence of Lipschitz potentials $A^{*}_{n} (w) = -d(w,\gamma_n\cup \gamma).$ The support of the maximizing probability for $A^{*}_{n}$ is $\gamma_n\cup \gamma$.
Moreover
$$0=m(A^{*}_{n})=\max_{\nu \text{\, an invariant probability for $\sigma$}}
\int  A^{*}_{n}(w) \; d\nu(w).$$


Denote by $V^{*}_{n}$ a Lipschitz calibrated subaction for $A^{*}_{n}$ such that $V^{*}_{n}(w_0)=0$. In this way, for all $w$
$$ R^{*}_{n}(w) =  (V^{*}_{n} \circ\sigma-V^{*}_{n}-A^{*}_{n}) \, (w)\geq 0,$$
and for $w\in \gamma_n \cup \gamma$ we have that $R^{*}_{n}(w) =0.$

We know that $R^{*}_{n}$ is zero on the orbit $\gamma_n$, because $\gamma_{n}$ is included in the Masur set.

Note that we not necessarily have $R^{*}_{n}(w_0)=0$.

By construction, the Lipschitz constant for $A^{*}_{n}$ is $1$. This is also true for $V^{*}_{n}$. Hence the family of subactions $(V^{*}_{n})$ is a family of equicontinuous functions. Let us denote by $V^{*}$ any accumulation point for $(V^{*}_{n})$ for the $\mathcal{C}^0$-topology. Note that $V^{*}$ is also $1$-Lipschitz continuous. For simplicity we set
$$V^{*}=\lim_{k\to\infty}V^{*}_{n_{k}}.$$

We denote by $\Gamma$ the set which is the limit of the sets $\gamma_n$ (using the Hausdorff distance). $\gamma\cup \Gamma$ is a compact set. Note that $\Gamma$ is not a compact set, but the set of accumulation points for $\Gamma$ is the set $\gamma$. We now consider $A^{*}(w) = -d(w, \gamma\cup\Gamma)$.

As any accumulation point of $\Gamma$ is in $\gamma$, any maximizing probability for the potential $A^{*}$ has support in $\gamma$. On the contrary, the unique $\sigma$-invariant measure with support in $\gamma$ is maximizing for $A^{*}$.



Remember that for any $n$ we have $V^{*}_{n} (w_0) = 0$.We also claim that we have $A^{*}_{n_k} (w_0)\to 0$ and $V^{*}_{n_k} (\sigma(w_0)) \to 0,$ as $k \to \infty$.

For each fixed $w$ we set

$$ R^{*}_{n_k}(w) =  (V^{*}_{n_k} \circ\sigma-V^{*}_{n_k} -A^{*}_{n_k}) \, (w)\geq 0.$$
The right hand side terms converge (for the $\mathcal{C}^0$-topology) as $k$ goes to $+\infty$. Then $R^{*}_{n_{k}}$ converge, and we denote by $R^{*}$ its limit. Then  for every $w$ we have:
$$ R^{*}(w) =  (V^{*}\circ\sigma-V^{*}-A^{*}) \, (w)\geq 0.$$

This shows that $V^{*}$ is a subaction for $A^{*}$. Note that $R^{*}(w_0)=0$.
From the uniqueness of the maximizing probability for $A^{*}$ we know that there exists a unique calibrated subaction for $A^{*}$ (up to an additive constant).

Consider a fixed $w$ and its two preimages $w_a$ and $w_b$.
For any given $n$, one of the two possibilities occur $R^{*}_{n}(w_a) =0$ or $ R^{*}_{n}(w_b)=0$, because $V^{*}_{n}$ is calibrated for $A^{*}_{n}$.

Therefore, for an infinite number of values $k$ either $R^{*}_{n_k}(w_a) =0$ or $ R^{*}_{n_k}(w_b)=0$.

In this way the limit of $V^{*}_{n_k}$ is unique (independent of the convergent subsequence)  and equal to $V^{*}$, the calibrated subaction for $A^{*}$ (such that $V^{*}(w_0)=0$).

Therefore,
$$ R^{*}(w_0) =  (V^{*}\circ\sigma-V^{*}-A^{*}) \, (w_0)= 0,$$
and $V^{*}$ is  a calibrated subaction for $A^{*}(w) = d(w, \gamma \cup \Gamma)$.

\end{example}

\vspace{0.2cm}

\bigskip

\section{Generic continuity of the Aubry set.}

In this section and in the next  we will present the proof of the generic properties we mention before.

In order to do that we will need several general properties in Ergodic Optimization.

Remember that we will denote the action of the shift in the points $x$ by $T$, and, we leave $\sigma$ for the action of the shift in the coordinates $w$.

We will present our main results in great generality. First, in this section, we analyze the main properties of sub-actions and its dependence on the potential $A$.

First we will present the main definitions we will consider here.

$\cF\subset C^0(\Sigma,\re)$ denotes a complete metric space with a (topology finer than)
metric larger than $d_{C^0}(f,g)=\lV f-g\rV_0:=\sup_{x\in \Sigma}\lv f(x)-g(x)\rv$; (for instance,
H\"older functions, Lispchitz functions, etc)
AND such that
\begin{equation}\label{noanal}
\forall K \subset \Sigma \text{ compact },\,\, \exists \psi\in\cF \text{ s.t. }
\psi\le 0, \;\;[\psi=0]=\{x\,|\, \psi(x)=0\}=K.
\end{equation}

Given $A\in \cF$ and   $F$ a calibrated sub-action for $A$, remember that
its {\it error} is denoted by  $R=R_A:\Sigma\to[0,+\infty[$:
$$
R(x):= F(T(x))-F(x)-A(x)+m_A\ge 0.
$$

$\cS(A)$ denotes the set of Holder calibrated sub-actions (it is not empty \cite{CG} \cite{CLT})

Given $A$, the {\it Ma\~n\'e action potential} is:
$$
S_A(x,y):=\lim_{\e\to 0}
\Big[\sup\Big\{\sum_{i=0}^{n-1}\big[A(T^i(z))-m_A\big]\;\Big|\;n\in\na,\;T^n(z)=y,\;d(z,x)<\e\Big\} \Big].
$$

Given $x$ and $y$ the above value describe the $A$-cost of going from $x$ to $y$ following the dynamics.

The {\it Aubry set} is
$\cA(A):=\{\,x\in \Sigma\,|\, S_A(x,x)=0\,\}$.

The terminology is borrowed from the Aubry-Mather Theory \cite{CI}.

For any $x \in\cA(A)$,  we have that $S_A(x,.)$  is a sub-action (in particular, in this case, $S_A(x,y) > -\infty$, for
any $y$), see Proposition 23 in \cite{CLT}.

The set of maximizing measures is
$$
\cM(A) := \{\,\mu\in\cM(T)\;|\; \int A\,d\mu= m_A\,\}.
$$



If $F\in C^\a(\Sigma,\re)$ is a H\"older function define
$$
|F|_\alpha:= \sup_{x\ne y}\frac{\lv F(x)-F(y)\rv}{d(x,y)^\a}.
$$

 Define the {\it Ma\~n\'e set} as
 $$
 \N(A):=\bigcup_{F\in{\cS}(A)} I_F^{-1}\{0\},
 $$
 where the union is among all the $\a$-H\"older calibrated sub-actions $F$ for $A$ and
 $$
 I_F(x) = \sum_{i=0}^\infty R_F(T^i(x)).
 $$

$I_F(x)$ is the deviation function we considered before.

For $A\in\cF$ define the {\it Mather set} as
  $$
  \M(A):=\bigcup_{\mu\in\cM(A)}\text{supp}(\mu).
  $$

 \bigskip

  The {\it Peierls barrier} is
\begin{align*}
h_A(x,y):&= \lim_{\e\to 0} \limsup_{k\to+\infty} S_A(x,y,k,\e),\\
\text{where}\,\, S_A(x,y,k,\e) &:=
\sup\Big\{\sum_{i=0}^{n-1}\big[A(T^i(z))-m_A\big]\;\Big|\; n\ge k,\;T^n(z)=y,\;d(z,x)<\e\Big\}.
\end{align*}

 Several properties of the Ma\~n\'e potential and the Peierls barrier are similar (but not all, see section 4 in \cite{GL3}). We will present proofs for one of them and the other case is similar.

 \begin{lemma}\label{LAS}\quad

\begin{enumerate}
 \item \label{mMA}

 If $\mu$ is a minimizing measure then
$$
\text{supp}(\mu)\subset \A(A)=\{\,x\in\Sigma\;|\;S_A(x,x)=0\;\}.
$$
\item\label{Sm0} $S_A(x,x)\le 0$, for every $x\in\Sigma$.
\item\label{SC0}
For any $z\in\Sigma$, the function $F(y)=h_A(z,y)$ is H\"older continuous.
\item\label{SC1}
If, $a\in\A(A)$, then, $h_A(a,x)=S_A(a,x)$, for all $x\in\Sigma$.

In particular, $F(y)=S_A(a,x)$ is continuous, if $a\in\A(A)$.

\item\label{SC2} If $S_A(w,y)=h_A(w,y)$ then the function
$F(y)= S_A(w,y)$ is continuous at $y$.

\item\label{SC3} If $S(x_0,T^{n_k}(x_0))=\sum_{j=0}^{n_k-1}A(T^j(x_0))$,
$\lim_k T^{n_k}(x_0)=b$ and $\lim_k n_k =+\infty$, then
$$\lim_k S(x_0,T^{n_k}(x_0)) = S(x_0,b)$$.

\end{enumerate}
 \end{lemma}

Item \eqref{mMA} follows from Atkinson-Ma\~n\'e's lemma which says that
if $\mu$ is ergodic for $\mu$-almost every $x$ and every $\e>0$, the set
$$
N(x,\e):=\Big\{\,n\in\na\;\Big |\; \lv\sum_{j=0}^{n-1}A(T^j(x))- n\int A\,d\mu\rv\le \e\;\Big\}
$$
is infinite (see Lemma 2.2 \cite{Man} (which consider non-invertible transformation, \cite{CI}, \cite{CLT} or \cite{GL2} for the proof). We will show bellow just the items which are not proved in the mentioned references.

The problem with the discontinuity of $F(y)=S_A(w,y)$ is when the maximum
is obtained at a {\it finite} orbit segment
(i.e. when $S_A(w,y)\ne h_A(w,y)$), the hypothesis  in item \ref{SC2}.

 \begin{proof}\quad

By adding a constant we can assume that $m_A=0$.

 {\bf (\ref{Sm0}).} Let $F$ be a continuous sub-action for $A$. Then
$$
-R_F= A+F -F\circ T \le 0.
$$
Given $x_0\in \Sigma$, let $x_k\in \Sigma$ and $n_k\in\na$ be such that
$T^{n_k}(x_k)=x_0$, $\lim_k x_k=x_0$ and
$$
S(x_0,x_0) = \lim_k \sum_{j=0}^{n_k-1}A(T^j(x_k)).
$$
 We have
\begin{align*}
 \sum_{j=0}^{n_k-1}A(T^j(x_k))
    &= \Big[\sum_{j=0}^{n_k-1}\big(A+F-F\circ T\big)(T^j(x_k))\Big] +F(x_0)-F(x_k)
\\
&\le F(x_0)-F(x_k).
\end{align*}

Then
\begin{align*}
 S(x_0,x_0)= \lim_k \sum_{j=0}^{n_k-1}A(T^j(x_k))\le \lim_k\big[ F(x_0)-F(x_k)\big] =0.
\end{align*}

 \bigskip

The proofs of (3) (4) (5) can be found in \cite{CLT} \cite{GL2}
\bigskip

  {\bf (\ref{SC3}).}  Let $\tau_k$ be the branch of the inverse of $T^{n_k}$ such that
                      $\tau_k(T^{n_k}(x_0))=x_0$. Let $b_k=\tau_k(b)$ for $k$ sufficiently large.
  Then, by the expanding property of the shift, there is $\lambda<1$, such that,
\begin{gather*}
d(x_0,b_k)\le \la^{n_k}\, d(T^{n_k}(x_0),b) \overset{k}\longrightarrow 0,
\\
\lv\sum_{i=0}^{n_k-1} A(T^i(x_0)) -\sum_{i=0}^{n_k-1} A(T^i(b_k))\rv
   \le \frac{\lV A\rV_\a}{1-\la^\a}\, d(T^{n_k}(x_0),b)^\a.
\end{gather*}
Write $Q:=\frac{|A|_\alpha}{1-\la^\a}$, then
\begin{align*}
 S(x_0,b)&\ge\limsup_k\sum_{i=0}^{n_k-1} A(T^i(b_k)) \\
         &\ge \limsup_k S(x_0,T^{n_k}(x_0)) - Q\, d(T^{n_k}(x_0),b)^\a
\\
         &\ge \limsup_k S(x_0,T^{n_k}(x_0)).
\end{align*}

Now for $\ell\in\na$ let $b_\ell\in \Sigma$  and  $m_\ell\in\na$ be such that
$\lim b_\ell = x_0$, $T^{m_\ell}(b_\ell)=b$ and
$$
 \lim_\ell\sum_{j=0}^{m_\ell-1} A(T^j(b_\ell)) = S(x_0,b).
$$
Let $\hat{\tau}_l$ be the branch of the inverse of $T^{m_\ell}$ such that $\hat{\tau}_l(b)=b_\ell$.
Let $x_\ell:=\hat{\tau}_l(T^{n_k}(x_0))$.
Then
\begin{align*}
d(x_\ell,x_0) &\le d(x_0,b_\ell)+d(b_\ell,x_\ell)
\\
&\le d(x_0,b_\ell)+\la^\ell\, d(T^{n_k}(x_0),b)
\overset{\ell}\longrightarrow 0.
\end{align*}
$$
\lv\sum_{j=0}^{m_\ell-1}A(T^j(x_\ell))-\sum_{j=0}^{m_\ell-1}A(T^j(b_\ell))\rv
\le Q\, d(T^{n_k}(x_0),b)^\a.
$$
Since $x_\ell\to x_0$ and $T^{m_\ell}(x_\ell)=T^{n_k}(x_0)$, we have that
\begin{align*}
 S(x_0,T^{n_k}(x_0)) &\ge \limsup_\ell \sum_{j=0}^{m_\ell -1}A(T^j(x_\ell))
\\
&\ge \limsup_\ell \sum_{j=0}^{m_\ell -1}A(T^j(b_\ell)) - Q\, d(T^{n_k}(x_0),b)^\a
\\
&\ge S(x_0,b) - Q\, d(T^{n_k}(x_0),b)^\a.
\end{align*}
And hence
$$
\liminf_k S(x_0,T^{n_k}(x_0)) \ge S(x_0,b).
$$

 \end{proof}

 \begin{pro}\label{AS}
 The Aubry set is
  $$
  \cA(A)=\bigcap_{F\in\cS(A)} I_F^{-1}\{0\},
  $$
 where the intersection is among all the $\a$-H\"older calibrated sub-actions for $A$.
 \end{pro}

\begin{proof}\quad

By adding a constant we can assume that $m_A=0$.

We first prove that $\cA(A)\subset\bigcap_{F\in\cS(A)} I_F^{-1}\{0\}$.

Let $F\in\cS(A)$ be a H\"older sub-action and $x_0\in\A(A)$.
Since $S_A(x_0,x_0)=0$ then there is $x_k\to x_0$ and $n_k\uparrow \infty$ such that
$\lim_k T^{n_k}(x_k)=x_0$ and $\lim_k\sum_{j=0}^{n_k-1}A(T^j(x_k))=0$.
If $m\in\na$ we have that
\begin{align}
 F(T^{m+1}(x_0))
&\ge F(T^m(x_0))+A(T^m(x_0))
\notag\\
&\ge F(T^{m+1}(x_k)) +\sum_{j=m+1}^{n_k+m-1}A(T^j(x_k))+A(T^m(x_0))
\notag\\
&\ge F(T^{m+1}(x_k)) +\sum_{j=0}^{n_k-1}A(T^j(x_k))
- \sum_{j=0}^{m}|A(T^j(x_k))-A(T^j(x_0))|
\label{rhs1}
\end{align}
When $k\to\infty$ the right hand side of \eqref{rhs1} converges to
$F(T^{m+1}(x_0))$, and hence all those inequalities are equalities.
Therefore $R_F(T^m(x_0))=0$ for all $m$ and hence $I_F(x_0)=0$.

Now let $x_0\in\bigcap_{F\in\cS(A)} I_F^{-1}\{0\}$.
Sinc $\Sigma$ is compact there is $n_k\overset{k}\to +\infty$ such that
the limits $b=\lim_k T^{n_k}(x_0)\in\Sigma$ and $\mu=\lim_k\mu_k\in\cM(T)$,
$\mu_k:=\frac1{n_k}\sum_{i=0}^{n_k-1}\delta_{T^i(x_0)}$ exist and $b\in\text{supp}(\mu)$.
Let $G$ be a H\"older calibrated sub-action. For $m\ge n$ we have
\begin{align*}
G(T^n (x_0))+S_A(T^n(x_0),T^m (x_0))
&\ge G(T^n(x_0))+\sum_{j=n}^{m-1}A(T^j(x_0))
\\
&= G(T^m(x_0)) \qquad\qquad\text{[because $I_G(x_0)=0$\,]}
\\
&\ge G(T^n (x_0))+S_A(T^n(x_0),T^m (x_0)).
\end{align*}
Then they are all equalities and hence for any $m\ge n$
$$
S_A(T^n(x_0),T^m (x_0))=\sum_{j=n}^{m-1}A(T^j(x_0)).
$$
Since
\begin{align*}
 0=\lim_k\frac 1{n_k} S_A(T^n(x_0),T^m (x_0))=
\lim_k\frac 1{n_k}\sum_{j=n}^{m-1}A(T^j(x_0))
=\int A\,d\mu,
\end{align*}
$\mu$ is a minimizing measure. By lemma~\ref{LAS}.\eqref{mMA},
$b\in\A(A)$.

Let $F:\Sigma\to\re$ be $F(x):=S_A(b,x)$. Then $F$  is a H\"older calibrated sub-action.
By hypothesis $I_F(x_0)=0$ and then
\begin{align*}
F(T^n(x_0))&= F(x_0)+ S_A(x_0,T^{n_k}(x_0)).\\
S_A(b,T^{n_k}(x_0)) &= S_A(b,x_0)+ S_A(x_0,T^{n_k}(x_0)).
\end{align*}
By lemma~\ref{LAS}.\eqref{SC1} and Lemma~\ref{LAS}.\eqref{SC3}, taking the limit on $k$ we have that
\begin{gather*}
0=S_A(b,b)=S_A(b,x_0)+S_A(x_0,b)=0.
\\
0\ge S_A(x_0,x_0)\ge S_A(x_0,b)+S_A(b,x_0)=0
\end{gather*}
Therefore $x_0\in\A(A)$.

\end{proof}

We want to show the following result which will require several preliminary results.

    \begin{thm}\label{TR}
     The set
     \begin{equation}\label{eTR}
     \cR:=\{\,A\in C^\a(\Sigma,\re)\;|\; \cM(A)=\{\mu\},\; \A(A)=\text{supp}(\mu)\,\}
     \end{equation}
     contains a residual set in $C^\a(\Sigma,\re)$.
    \end{thm}

 The proof of the bellow lemma (Atkinson-Ma\~n\'e) can be found in \cite{Man} and \cite{CI}.

 \begin{lemma} \label{FancyBirk}
 Let $(X, \fB, \nu)$ be a probability space, $f$ an ergodic measure
 preserving map and $F:X\to \re$ an integrable function. Given $A\in
 \fB$
 with $\nu(A)>0$ denote by $\hat{A}$ the set of points $p\in A$ such
 that for all
 $\e>0$ there exists an integer $N>0$ such that $f^N(p)\in A$ and
 $$
 \Big| \sum_{j=0}^{N-1} F\bigl( f^j(p) \bigr) - N\,
  \int F \, d\nu\, \Big|
 < \e \, .
 $$
 Then $\nu(\hat{A})=\nu(A)$.
 \end{lemma}

 \begin{corollary}\label{C.FancyBirk}
 If besides the hypothesis of lemma~\ref{FancyBirk}, $X$ is
 a complete separable metric space,
 and $\fB$ is its Borel $\sigma$-algebra, then for a.e.
 $x\in X$
 the following property holds: for all $\e>0$ there exists $N>0$ such
 that
 $d( f^N(x), x)<\e$ and
 $$
 \lv {\sum_{j=0}^{N-1}} F\bigl( f^j(x)\bigr) - N \int F\, d\nu\rv<\e
 $$
 \end{corollary}

 \begin{proof}
  Given $\e>0$ let $\{ V_n(\e)\}$ be a countable basis of
 neighborhoods with
 diameter $<\e$ and let $\hat{V}_n$ be associated to $V_n$ as in
 lemma~\ref{FancyBirk}.
 Then the full measure subset $\underset m\cap \underset n\cup
 \hat{V}_n(\frac1m)$
 satisfies the required property.
 \end{proof}

    \begin{lemma}\label{PR}
     Let $\cR$ be as in Theorem~\ref{TR}.
     Then if $A\in\cR$, $F\in\cS(A)$  we have
     \begin{enumerate}
      \item\label{PR0}  If $a,\,b\in\A(A)$ then $S_A(a,b)+S_A(b,a)=0$.
      \item\label{PR1}  If $a\in\A(A)=\text{supp}(\mu)$ then $F(x)=F(a)+S_A(a,x)$ for all $x\in\Sigma$.
      \end{enumerate}

    \end{lemma}

    \begin{proof}\quad

      {\bf (\ref{PR0}).} Let $a,\,b\in\A(A)=\text{supp}(\mu)$. Since $\mu$ is ergodic, by Corollary~\ref{C.FancyBirk}
       there are sequences $\a_k\in\Sigma$, $m_k\in\na$ such that $\lim_k m_k=\infty$, $\lim_k \a_k=a$,
       $\lim_k d(T^{m_k}(\a_k),\a_k) = 0$,
       $$
       \sum_{j=0}^{m_k-1} A(T^j(\a_k))\ge \frac 1k,
       \qquad\text{ and  writing }\qquad
        \mu_k:=\frac 1{m_k} \sum_{j=1}^{m_k-1} \de_{T^{m_k}(\a_k)},
       \quad \lim_k\mu_k =\mu.
       $$
       Since $b\in\text{supp}(\mu)$ there are $n_k\le m_k$ such that $\lim_k T^{n_k}(\a_k)=b$.

       Let $\tau_k$ be the branch of the inverse of $T^{n_k}$ such that $\tau_k(T^{n_k}(\a_k))=\a_k$.
       Let $b_k:=\tau_k(b)$. Then $T^{n_k}(b_k)=b$ and
       \begin{align*}
        d(b_k,a) &\le d(b_k,\a_k)+d(\a_k,a)   \\
                 &\le \la^{n_k}\,d(b,T^{n_k}(\a_k))+ d(\a_k,a)   \\
                 &\le d(b,T^{n_k}(\a_k))+d(\a_k,a) \overset{k}\longto 0.
       \end{align*}
       We have that
       $$
       \lv\sum_{j=0}^{n_k-1}A(T^j(b_k))-\sum_{j=0}^{n_k-1}A(T^j(\a_k))\rv
       \le \frac{\lV A\rV_\a}{1-\la^\a}\;d(T^{n_k}(\a_k),b)^\a.
       $$
       \begin{align*}
        S(a,b) &\ge \limsup_k \sum_{j=0}^{n_k-1} A(T^j(b_k)) \\
               &\ge \limsup_k \sum_{j=0}^{n_k-1} A(T^j(\a_k)) - Q\, d(T^{n_k}(\a_k),b)^\a.
       \end{align*}

       Let $\tau_k$ be the branch of the inverse of $T^{m_k-n_k}$ such that
       $\tau_k(T^{m_k}(\a_k))=T^{n_k}(\a_k)$. Let $a_k:=\tau_k(a)$ Then
       $T^{m_k-n_k}(a_k)=a$ and
       \begin{align*}
        d(b,a_k) &\le d(b,T^{n_k}(\a_k))+d(T^{n_k}(\a_k),a_k)  \\
                 &\le d(b,T^{n_k}(\a_k)) +\la^{m_k-n_k}\, d(T^{m_k}(a_k),a) \\
                 &\le d(b,T^{n_k}(\a_k)) +                d(T^{m_k}(a_k),a)
                 \overset{k}\longto 0.
       \end{align*}
       Also
       $$
       \lv\sum_{j=0}^{m_k-n_k-1}A(T^j(a_k))-\sum_{j=n_k}^{m_k-1}A(T^j(\a_k))\rv
       \le \frac{\lV A\rV_\a}{1-\la^\a}\, d(a,T^{m_k}(\a_k)).
       $$
       \begin{align*}
        S(a,b) &\ge \limsup_k \sum_{j=0}^{m_k-n_k-1}A(T^j(a_k)) \\
               &\ge \limsup_k \sum_{j=n_k}^{m_k-1} A(T^j(\a_k) -Q\; d(a,T^{n_k}(\a_k))
       \end{align*}

       Therefore
        \begin{align*}
         0\ge S(a,a) &\ge S(a,b)+S(b,a) \\
                     &\ge \limsup_k \sum_{j=0}^{n_k-1}A(T^j(\a_k))
                      +\limsup_k\sum_{j=n_k}^{m_k-1}A(T^j(\a_k))  \\
                     &\ge \limsup_k \Big[\sum_{j=0}^{n_k-1}A(T^j(\a_k)) + \sum_{j=n_k}^{m_k-1}A(T^j(\a_k))\Big]\\
                     &\ge \limsup_k \frac 1k \\
                     &\ge 0.
        \end{align*}

      {\bf (\ref{PR1}).} We first prove that if for some $x_0\in\Sigma$ and $a\in\A(A)$ we have
      \begin{equation}\label{somea}
       F(x_0)=F(a)+S_A(a,x_0),
      \end{equation}
      then equation~\eqref{somea} holds for every $a\in\A(A)$.
       If $b\in\A(A)$, using item~\ref{PR0} we have that
      \begin{align*}
       F(x_0) &\ge F(b)+S(b,x_0) \\
       &\ge F(a)+S_A(a,b)+S_A(b,x_0) \\
       &\ge F(a)+S_A(a,b)+S_A(b,a)+S_A(a,x_0)\\
       &=F(a)+S_A(a,x_0)\\
       &=F(x_0).
      \end{align*}
      Therefore $F(x_0) = F(b)+S(b,x_0)$.

      It is enough to prove that given any $x_0\in\Sigma$ there is $a\in\A(A)$ such that
      the equality~\eqref{somea} holds. since $F$ is calibrated there are
       $x_k\in\Sigma$ and  $n_k\in\na$ such that $T^{n_k}(x_k)=x_0$, $\exists\lim_k x_k=a$ and
       for every $k\in\na$,
       $$
       F(x_0) = F(x_k) + \sum_{j=0}^{n_k-1}A(T^j(x_k)).
       $$
       We have that
       \begin{align*}
        S_A(a,x_0)&\ge \limsup_k\sum_{j=0}^{n_k-1}A(T^j(x_k))\\
        &= \limsup_k F(x_0)-F(x_k) \\
        &= F(x_0)-F(a) \\
        &\ge S(a,x_0).
       \end{align*}
       Therefore equality~\eqref{somea} holds.

       It remains to prove that $a\in\A(A)$, i.e. that $S_A(a,a)=0$.
       We can assume that the sequence $n_k$ is increasing.
       Let $m_k=n_{k+1}-n_k$. Then $T^{m_k}(x_{k+1})=x_k$. Let $\tau_k$ be the branch of the inverse
       of $T^{m_k}$ such that $\tau_k(x_k)=x_{k+1}$ and $a_{k+1}:=\tau_k(a)$.
       We have that
       $$
       \lv \sum_{j=0}^{m_k-1}A(T^j(a_{k+1})) -\sum_{j=0}^{m_k-1}A(T^j(x_{k+1}))\rv
       \le \frac{\lV A\rV_\a}{1-\la^{\a}}\; d(a,x_k)^\a.
       $$
       Since $x_k\to a$ we have that
       \begin{align*}
        d(a_{k+1},a)&\le d(a_{k+1},x_{k+1})+d(x_{k+1},a) \\
        &\le \la^{m_k}\,d(x_k,a)+d(x_{k+1},a) \\
        &\le d(x_k,a)+d(x_{k+1},a) \overset{k}\longrightarrow 0.
       \end{align*}
       Therefore
       \begin{align*}
        0\ge S_A(a,a) &\ge \limsup_k \sum_{j=0}^{m_k-1}A(T^j(a_{k+1}))\\
        &\ge \limsup_k \sum_{j=0}^{m_k-1}A(T^j(x_{k+1})) - Q\; d(a,x_k)^\a\\
        &= \limsup_k F(x_k)-F(x_{k+1})- Q\; d(a,x_k)^\a
        \\
        &=0.
       \end{align*}

    \end{proof}

The above result (2) is true for $F$ only continuous.

\begin{corollary}\label{CPR}
     Let $\cR$ be as in Theorem~\ref{TR}.
     Then if $A\in\cR$, $F\in\cS(A)$  we have
     \begin{enumerate}
      \item\label{PR2}  If $x\notin \A(A)$ then $I_F(x)>0$.
      \item\label{PR3}  If $x\notin \A(A)$ and $T(x)\in\A(A)$ then $R_F(x)>0$.
     \end{enumerate}
    \end{corollary}

\begin{proof}\quad

     {\bf (\ref{PR2}).} By lemma~\ref{PR}.\eqref{PR1} modulo adding a constant there
     is only one H\"older calibrated sub-action $F$ in $\cS(A)$. Then by proposition~\ref{AS},
     $\A(A)=[I_F=0]$. Since $I_F\ge 0$, this proves item~\ref{PR2}.

     {\bf (\ref{PR3}).}
      Since $T(x)\in\A(A)$
      $$
       I_F(x) =\sum_{n\ge 1} R_F(T^n(x)) =0.
      $$
      Since $x\notin\A(A)$, by item~\ref{PR1} and proposition~\ref{AS},
      $\A(A)=[I_F=0]$. Then
      $$
      I_F(x) =\sum_{n\ge 0} R_F(T^n(x)) >0.
      $$
      Hence $R_F(x)>0$.
\end{proof}

   \bigskip
   \bigskip
   \bigskip

\begin{lemma}\label{LPL}\quad
\begin{enumerate}
\item\label{lpl1} $A\mapsto m_A$ has Lipschitz constant 1.

\item\label{lpl5} Fix $x_0 \in \Sigma$.
The set $\cS(A)$ of $\a$-H\"older calibrated sub-actions $F$ for $A$  with $F(x_0)=0$
is an equicontinuous family. In fact
$$
 \sup_{F\in\cS(A)}|F|_\alpha <\infty.
$$

\item\label{lpl55} The set $\cS(A)$ of $\a$-H\"older continuous calibrated sub-actions is closed under the $C^0$ topology.

\item\label{lpl6} If  $\#\cM(A)=1$, $A_n\overset{n}\to A$ uniformly, $\sup_n   |A_n|_\alpha<\infty$, and, $F_n\in\cS(A_n)$, \\
            then, $\lim_n F_n = F$ uniformly.

\item\label{lpl27} $A\le B \quad\&\quad m_A=m_B \quad \Longrightarrow \quad S_A\le S_B$.

\item\label{lpl2} $\limsup\limits_{B\to A}\N(B) \subseteq \N(A)$, where
$$
\limsup\limits_{B\to A}\N(B) =\big\{\, \lim_n x_n \;\big|\; x_n\in \N(B_n), \;B_n \overset{n}{\to} A, \,\exists\lim_n x_n\,\big\}
$$

\item\label{lpl3} If $A\in\cR$
                             then
                             $$
                             \lim\limits_{B\to A}d_H(\cA(B),\cA(A))=0,
                             $$ \\
where $d_H$ is the Hausdorff distance.

\item\label{lpl8}  If $A\in\cR$ with $\cM(A)=\{\,\mu\,\}$ and $\nu_B\in \cM(B)$ then
$$
\lim_{B\to A}d_H\big(\text{supp}(\nu_B),\text{supp}(\mu)\big)= 0.
$$
\item\label{lpl9} If $A\in\cR$ then
$$
\lim_{B\to A}d_H\big(\M(B),\A(A)\big)= 0.
$$

\end{enumerate}
\end{lemma}

If $X, \, Y$ are two metric spaces and $\F:X\to 2^Y=\bP(Y)$ is a set valued function,
define
\begin{align*}
\limsup_{x\to x_0} \F(x) &= \bigcap_{\e>0}\;\;\bigcap_{\de>0}\;\;\bigcup_{d(x,x_0)<\de}V_\e(\F(x)),
\\
\liminf_{x\to x_0} \F(x) &= \bigcap_{\e>0}\;\;\bigcup_{\de<0}\;\;\bigcap_{d(x,x_0)<\de} V_\e(\F(x)),
\end{align*}
where
$$
V_\e(C)=\bigcup_{y\in C} \{\, z\in Y \;|\; d(z,y) < \e\,\}.
$$

\begin{proof}\quad

\noindent{\bf (\ref{lpl1}).} We have that $A\le B+\lV A-B\rV_0$, then
\begin{align*}
\int A\;d\mu &\le \int B\;d\mu+\lV A-B\rV_0, \qquad \forall \mu\in\cM(T),\\
\int A\;d\mu &\le \sup_{\mu\in\cM(T)}\int B\;d\mu+\lV A-B\rV_0 = m_B+\lV A-B\rV_0,\\
m_A&\le m_B+\lV A-B\rV_0.
 \end{align*}
Similarly $m_B\le m_A+\lV A-B\rV_0$ and then
$\lv m_A-m_B\rv\le \lV A-B\rV_0$.

See also \cite{Jenkinson1} and \cite{CLT} for a proof.

\bigskip

\noindent{\bf (\ref{lpl5}).} Let $\e>0$ and $0<\la<1$ be such that for any $x\in\Sigma$ there is an
inverse branch $\tau$ of $T$ which is defined on the ball
$B(T(x),\e):=\{\,z\in \Sigma\;|\; d(z,T(x))< \e\,\}$, has Lipschitz constant $\la$ and $\tau(T(x))=x$.

Let $F\in\cS(A)$. Let
$$
K:=|F|_\alpha:=\sup_{d(x,y)<\e}\frac{\lv F(x)-F(y)\rv}{d(x,y)^\a},
\qquad
a:=|A|_\alpha:=\sup_{d(x,y)<\e}\frac{\lv A(x)-A(y)\rv}{d(x,y)^\a}
$$
be H\"older constants for $F$ and $A$. Given $x, \; y \in \Sigma$ with $d(x,y)<\e$
let $\tau_i$, $i=1,\ldots,m(x)\le M$ be the inverse branches for $T$  about $x$ and let
$x_i=\tau_i(x)$, $y_i=\tau_i(y)$.
We have that
\begin{gather*}
 \lv F(x_i)-F(y_i)\rv K; \la^\a\; d(x,y)^\a, \qquad
 \lv A(x_i)-A(y_i)\rv \a\; \la^\a\; d(x,y)^\a. \\
F(x_i)+A(x_i)\le F(y_i)+A(y_i)+(K+a)\;\la^\a\;d(x,y)^\a,
\\
\max_i \big[F(x_i)+A(x_i)-m_A\big]\le \max_i \big[F(y_i)+A(y_i)-m_A\big]+(K+a)\;\la^\a\;d(x,y)^\a,
\\
F(x)\le F(y) + (K+a)\; \la^\a\; d(x,y)^\a,
\end{gather*}
Then $|F|_\alpha\le \la^\a\,(  |F|_\alpha+ |A|_\alpha)$
and hence
\begin{equation}\label{lpl5e}
|F|_\alpha \le \frac{\la^\a}{(1-\la^\a)}\, |A|_\alpha.
\end{equation}
This implies the equicontinuity of $\cS(A)$.

The proof of the above result could be also get if we just assume that $F$ is continuous.

\noindent{\bf (\ref{lpl55}).} It is easy to see that uniform limit of calibrated sub-actions is a sub-action,
and it is calibrated because the number of inverse branches of $T$ is finite, i.e.
$\sup_{y\in\Sigma} \#T^{-1}\{y\}<\infty$. By \eqref{lpl5} all $C^\a$ calibrated sub-actions have a common
H\"older constant, the uniform limits of them have the same H\"older constant.

\noindent{\bf (\ref{lpl6}).}  The family $\{ F_n\}$ satisfies $F_n(x_0)=0$ and by inequality \eqref{lpl5e}
$$
 |F_n|_\alpha<\frac{\la^\a}{(1-\la^\a)}\;\sup_n |A_n|_\alpha <\infty.
$$
Hence $\{ F_n\}$ is equicontinuous. By Arzel\'a-Ascoli theorem it is enough to prove
that there is a unique $F(x)=S_A(x_0,x)$ which is the limit of any convergent subsequence
of $\{ F_n\}$. Since $\sup  |A_n|_\alpha<\infty$, by inequality \eqref{lpl5e}, any such limit
is $\a$-H\"older. Since by lemma~\ref{PR}.\eqref{PR1},
$\cS(A)\cap[F(x_0)=0]=\{\,F(x)= S_A(x_0,x)\,\}$, it is enough to prove
that any limit of a subsequence of $\{ F_n\}$ is a calibrated sub-action. But this follows form the
continuity of $A\mapsto m_A$, the equality
$$
F_n (x) =\max_{T(y)=x} F_n(y)+A_n(y)-m_{A_n}
$$
and the fact $\sup\limits_{x\in\Sigma} \#(T^{-1}\{x\})<\infty$.

\bigskip

{\bf(\ref{lpl27}).}
The proof
follows from the expression
$$
S_A(x,y):=\lim_{\e\to 0}
\Big[\sup\Big\{\sum_{i=0}^{n-1}\big[A(T^i(z))-m_A\big]\;\Big|\;n\in\na,\;T^n(z)=y,\;d(z,x)<\e\Big\} \Big].
$$

{\bf(\ref{lpl2}).}  Let $x_n\in B_n\to A$ be such that $x_n\to x_0$. Let $F_n\in \cS(A)$ be such that
 $I_{F_n}(x_n)=0$. Adding a constant we can assume that $F_n(x_0)=0$ for all $n$. By \eqref{lpl5},
taking a subsequence we can assume that $\exists F = \lim_n F_n$ in the $C^0$ topology.
Then $F$ is a $C^\a$ calibrated sub-action for $A$. Also $R_{F_n}\to R_F$ uniformly and there is
a common H\"older constant $C$ for all the $R_{F_n}$. We have that
\begin{align*}
| R_{F_n}(T^k(x_n))-& R_F(T^k(x_0))|\le\\
     &\lv R_{F_n}(T^k(x_n))-R_{F_n}(T^k(x_0))\rv + \lv R_{F_n}(T^k(x_0))-R_F(T^k(x_0))\rv
     \\
     &\le C\; d(T^k(x_n),T^k(x_0))^\a + \lV R_{F_n}-R_F\rV \overset{n}\longrightarrow 0
\end{align*}
Since for all $n,\,k$, $R_{F_n}(T^k(x_n))=0$, we have that $R_F(T^k(x_0))=0$ for any $k$. Hence
$I_F(x_0)=0$ and then $x_0\in\N(A)$.

\bigskip

{\bf(\ref{lpl3}).} By Lemma~\ref{PR}.\eqref{PR1},
 there is only one calibrated sub-action modulo adding a constant.
Then by Proposition~\ref{AS},  $\A(A)=\N(A)$. Then by \eqref{lpl2} $\limsup_{B\to A}\A(B)\subset\A(A)$.
It is enough to prove that for any $x_0\in\A(A)$ and $B_n\to A$, there is $x_n\in\A(B_n)$
such that $\lim_n x_n = x_0$. Let $\mu_n\in\cM(B_n)$. Then $\lim_n\mu_n=\mu$ in the weak* topology.
Given $x_0\in\A(A)=\text{supp}(\mu)$ we have that
$$
\forall \e>0\quad\exists N=N(\e)>0\quad\forall n\ge N\;:\quad\mu_n(B(x_0,\e))>0.
$$
We can assume that for all $m\in\na$, $N(\frac 1m)<N(\frac 1{m+1})$.
For $N(\frac 1m)\le n<N(\frac 1{m+1})$ choose $x_n\in\text{supp}(\mu_n)\cap B(x_0,\frac 1m)$.
Then $x_n\in \A(B_n)$ and $\lim_n x_n=x_0$.
.

\bigskip

{\bf(\ref{lpl8}).}  For any $B\in\cF$ We have that
$$
\text{supp}(\nu_B)\subseteq \A(B)\subseteq\N(B).
$$
By item~\ref{lpl3},
$$
\limsup_{B\to A}\text{supp}(\nu_B)\subseteq\A(A)=\text{supp}(\mu).
$$
It remains to prove that
$$
\liminf_{B\to A}\text{supp}(\nu_B)\supseteq \text{supp}(\mu).
$$
 But this follows
from the convergence $\lim_{B\to A}\nu_B=\mu$ in the weak* topology.

\bigskip

{\bf(\ref{lpl9}).}  Write $\cM(A)=\{\,\mu\,\}$. By items \eqref{lpl3} and \eqref{lpl8} we have that
\begin{gather*}
\limsup_{B\to A}\M(B)\subseteq \limsup_{B\to A}\A(B)\subseteq \A(B),
\\
\A(A)=\text{supp}(\mu)\subseteq\liminf_{B\to A}\M(B).
\end{gather*}

\end{proof}

  \bigskip
 \bigskip

 {\bf Proof of Theorem~\ref{TR}.}

 The set
 $$
 \cD:=\{\, A\in\cF\,|\,\#\cM(A)=1\,\}
 $$
 is dense (c.f. \cite{CLT}). We first prove that $\cD\subset\ov{\cR}$, and hence that $\cR$ is dense.

 Given $A\in\cD$ with $\cM(A)=\{\mu\}$ and $\e>0$, let $\psi\in\cF$ be such that $\lV\psi\rV_0+\lV\psi\rV_\a< \e$
$\psi\le 0$, $[\psi=0]=\text{supp}(\mu)$.
It is easy to see that $\cM(A+\psi)=\{\mu\}=\cM(A)$.
 Let $x_0\notin \text{supp}(\mu)$.
Given $\de>0$, write
 $$
 S_A(x_0,x_0;\de):=\sup\Big\{\, \sum_{k=0}^{n-1} A(T^k(x_n))\;|\; T^n(x_n)=x_0,\; d(x_n,x_0)<\de\,\Big\}.
 $$

If  $T^n(x_n) = x_0$ is such that $d(x_n,x_0)<\de$ then
\begin{align*}
 \sum_{k=0}^{n-1} (A+\psi)(T^k(x_n)) &\le S_A(x_0,x_0;\de) +\sum_{k=0}^{n-1}\psi(T^k(x_n)) \\
 &\le S_A(x_0,x_0;\de) + \psi(x_n).
\end{align*}
Taking $\limsup_{\de\to 0}$,
$$
S_{A+\psi}(x_0,x_0)\le S_A(x_0,x_0)+\psi(x_0)\le\psi(x_0)<0.
$$
Hence $x_0\notin\A(A+\psi)$. Since by lemma~\ref{LAS}.\eqref{mMA}, $\text{supp}(\mu)\subset\A(A+\psi)$,
then $\A(A+\psi)=\text{supp}(\mu)$  and hence $A+\psi\in\cR$.

Let
$$
\cU(\e):=\{\, A\in\cF\;|\; d_H(\A(A),\M(A))<\e\,\}.
$$
From the triangle inequality
$$
d_H(\A(B),\M(B)) \le d_H(\A(B),\A(A))+d_H(\A(A),\M(B))
$$
and items \eqref{lpl3} and \eqref{lpl9} of lemma~\ref{LPL},
we obtain  that  $\cU(\e)$ contains a neighborhood of $\cD$.
Then the set
$$
\cR = \bigcap_{n\in\na}\cU(\tfrac 1n)
$$
contains a residual set.

\qed

\bigskip
\bigskip

\section{Duality.}\label{SSD}

In this section we have to consider properties for $A^*$ which depends of the initial potential $A$.

We will consider now the specific example described before.
We point out that the results presented bellow should hold in general for natural extensions.

We will assume that $T$ and $\si$ are topologically mixing.

Remember that
$\T:\hat{\Sigma}\to\hat{\Sigma}$,
\begin{align*}
 \T (x,\om)&=(T(x),\tau_x(\om))\;, &
\T^{-1}(x,\om)&=(\tau_\om(x),\si(\om))
&&
\end{align*}

 Given $A\in\cF$ define $\De_A:\Sigma\times\Sigma\times\Sigma\to \re$
 as
 $$
 \De_A(x,y,\om):=\sum_{n\ge 0}A(\tau_{n,\om}(x))-A(\tau_{n,\om}(y))
 $$
 where
 $$
 \tau_{n,\om}(x)=\tau_{\si^n\om}\circ\tau_{\si^{n-1}\om}\circ\cdots\circ\tau_\om(x).
 $$

 Fix $\ox\in\Sigma$ and $\ow\in\Si$.

The involution $W$-kernel can be defined as
 $W_A:\Sigma\times\Si\to\re$, $W_A(x,\om)=\De_A(x,\ox,\om)$.
Writing $\A:=A\circ\pi_1:\T\to\re$, we have that
$$
W(x,\om)=\sum_{n\ge 0}\A(\T^{-n}(x,\om))-\A(\T^{-n}(\ox,\om))
$$

We can get the {\it dual function}  $A^*:\Si\to\re$ as
$$
A^*(\om):=(W_A\circ \T^{-1}- W_A+ A\circ\pi_1)(x,\om).
$$

 Remember that we consider here the metric on $\Si$ defined by
 $$
 d(\om,\nu):= \la^N,\qquad N:=\min\{\,k\in\na\,|\,\om_k\ne \nu_k \,\}
 $$
 Then $\la$ is a Lipschitz constant for both $\tau_x$, and $\tau_\om$
 and also for $\T\vert_{\{x\}\times\Si}$ and $\T^{-1}\vert_{\Sigma\times\{\om\}}$


Write $\cF:=\CS$ and $\cF^*:=\CS$.
Let $\cB$ and $\cB^*$ be the set of coboundaries
 \begin{align*}
\cB:&=\{\,u\circ T - u \;|\; u\in\CS\,\}, \\
\cB^*:&=\{\,u\circ \si - u \;|\; u\in\CS\,\}.
 \end{align*}

Remember that
 $$
 ||z||_\alpha= \lV z\rV_0+ |z|_\alpha.
 $$

 We also use the notation $[z]_\alpha = ||z||_\alpha$.

 \begin{lemma}\label{LCob}\quad
 \begin{enumerate}
 \item\label{lcob.1} $z\in\cB \qquad \iff \qquad z\in\CS \quad \&\quad \forall\mu\in\cM(T),\quad \int z\,d\mu=0$.
 \item\label{lcob.2} The linear subspace $\cB\subset\CS$ is closed.
 \item\label{lcob.3} The function
                     $$
             \lB z+\cB\rB_\a =\inf_{b\in\cB}\lB z+b\rB_\a
                     $$
                     is a norm in $\cF/\cB$.
 \end{enumerate}
 \end{lemma}
 \begin{proof}\quad

 {\bf  (\ref{lcob.1}).}  This follows\footnote{Theorem 1.28 of R. Bowen \cite{Bow} asks for $T$ to be topologically mixing.} from \cite{Bow},  Theorem 1.28  (ii) $\then$ (iii).

 {\bf  (\ref{lcob.2}).}  We prove that the complement $\cB^c$ is open. If $z\in\CS\setminus\cB$, by
 item~\eqref{lcob.1}, there is $\mu\in\cM(T)$ such that $\int z\,d\mu\ne 0$.
 If $u\in\CS$ is such that
 $$
 \lV u-z\rV_0 <\frac 12\; \lv \int z\;d\mu\rv
 $$
 then $\int u\,d\mu\ne 0$
 and  hence $u\notin\cB$.

 {\bf  (\ref{lcob.3}).}  This follows from item \eqref{lcob.2}.

 \end{proof}

\begin{lemma}\label{LLL}\quad
\begin{enumerate}
 \item\label{lll.1}  If $A$ is $C^\a$ then $A^*$ is $C^\a$.
 \item\label{lll.2} The linear map $L:C^\a(\Sigma,\re)\to C^\a(\Si,\re)$ given by $L(A)=A^*$ is continuous.
 \item\label{lll.3} $\cB\subset\ker L$.
 \item\label{lll.4} The induced linear map $L:\cF/\cB \to \cF^*/\cB^*$ is continuous.
 \item\label{lll.5} Fix one $\oom\in\Si$. Similarly the corresponding linear map
       $L^*:\cF^*\to\cF$, given by
       \begin{align*}
        L^*(\psi) &= W_\psi^*\circ\T-W_\psi+\psi\circ\pi_2\\
         &=\sum_{n\ge 0}\Psi(\T^n(x,\oom))-\Psi(\T^n(Tx,\oom)) \\
        &=\Psi(x,\oom) + \sum_{n\ge 0}\Psi(\T^n(Tx,\tau_x\oom))-\Psi(\T^n(Tx,\oom)),
       \end{align*}
 with $\Psi= \psi\circ\pi_2$, is continuous and induces a continuous linear map \linebreak
 $L^*:\cF^*/\cB^*\to\cF/\cB$, which is the inverse of $L:\cF/\cB \to \cF^*/\cB^*$.

\end{enumerate}
\end{lemma}
\begin{proof}\quad

  {\bf (\ref{lll.1}) and (\ref{lll.2}).} We have that
 \begin{align*}
 A^*(\om)&=\sum_{n\ge 0}\A(\T^{-n}(\ox,\om))-\A(\T^{-n}(\ox,\si\,\om))
 \\
 &= A(\ox) +\sum_{n\ge 0}\A(\T^{-n}(\tau_\om\,\ox,\si\,\om))-\A(\T^{-n}(\ox,\si\,\om))
 \end{align*}
Since $d(\T^{-n}(\tau_\om\,\ox,\si\,\om),\T^{-n}(\ox,\si\,\om))\le\la^{n}\,d(\tau_\om\,\ox,\,\ox)\le \la^n$ and
$\lV \A\rV_\a=\lV A\circ\pi_1\rV_\a=\lV A\rV_\a$, we have that
\begin{align*}
\lV A^*\rV_0&\le \lV A \rV_0+\frac{\lV A\rV_\a}{1-\la^\a}.
\end{align*}
Also if $m:=\min\{\,k\ge 0 \,|\, w_k\ne\nu_k\,\}$
\begin{align*}
 A^*(\om)-A^*(\nu) =&\phantom{-}\sum_{n\ge m-1}\A(\T^{-n}(\tau_\om\,\ox,\si\,\om))-\A(\T^{-n}(\ox,\si\,\om)) \\
    &-\sum_{n\ge m-1}\A(\T^{-n}(\tau_\nu\,\ox,\si\,\nu))-\A(\T^{-n}(\ox,\si\,\nu))\\
\lv A^*(\om)-A^*(\nu)\rv \le &\;2\,\lV A\rV_\a\,\frac{\la^{(m-1)\a}}{1-\la^\a}
= \frac{ 2 \lV A\rV_\a\la^{-\a}}{1-\la^\a}\, d(\om,\nu)^\a.\\
\lV A^*\rV_\a \le &\;\frac{ 2 \lV A\rV_\a}{\la^\a(1-\la^{\a})}.
\end{align*}

{\bf (\ref{lll.3}).} If $u\in\cF$ and $\U:=u\circ \pi_1$ from the formula for $L$ (in the proof of item [\ref{lll.1}])
we have that
\begin{align*}
 L(u\circ T-u) &= \U(\T(\ox,\om))-\U(\T(\ox,\si \om)) \\
               &= u(T\, \ox) - u(T\,\ox) = 0.
\end{align*}

{\bf (\ref{lll.4}).} Item \eqref{lll.4} follows from items \eqref{lll.2} and \eqref{lll.3}.

{\bf (\ref{lll.5}).} We only prove that for any $A\in\cF$, $L^*( L(A)) \in A+\cB$. Write
\begin{align*}
  L^*(L(A)) &= \big(W^*_{A^*}\circ\T-W^*_{A^*} + \A^*\big) (\,\cdot\,, \oom) \\
            &= (W^*_{A^*}\circ\T-W^*_{A^*} + W_{A}\circ\T^{-1}-W_{A}) + A
\end{align*}
Write
\begin{equation}\label{eqB}
B:= W^*_{A^*}\circ\T-W^*_{A^*} + W_{A}\circ\T^{-1}-W_{A}.
\end{equation}
 Since $A, L^*(L(A))\in\cF=\CS$, then
$B\in\CS$.

 Following Bowen, given any $\mu\in\cM(T)$ we construct an associated measure $\nu\in\cM(\T)$.
 Given $z\in C^0(\Sigma\times\Si,\re)$ define $z^\sharp\in C^0(\Sigma,\re)$ as $z^\sharp(x):=z(x,\oom)$.
 We have that
 $$
 \lV (z\circ\T^n)^\sharp\circ T^m-(z\circ\T^{n+m})^\sharp\rv_0\le \text{var}_n z \overset{n}\longto 0,
 $$
 where
\begin{align*}
 \text{var}_n z &= \sup\{\,\lv z(a)-z(b)\rv\;|\;\exists x\in \Sigma,\; a,b\in\T^n(\{x\}\times\Si)\,\}\\
 &\le \sup\{\,| z(a)-z(b)|\;|\;d_{\Sigma\times\Si}(a,b)\le \la^n\;\}\overset{n}\longto 0,
\end{align*}
$$
d_{\Sigma\times\Si}=d_\Sigma\circ(\pi_1,\pi_1) + d_\Si\circ(\pi_2,\pi_2).
$$

Then
$$
\lv \mu((z\circ\T^n)^\sharp)-\mu((z\circ \T^{n+m})^\sharp\rv
=
\lv\mu((z\circ\T^n)^\sharp\circ T^m)-\mu((z\circ \T^{n+m})^\sharp\rv
\le
\text{var}_n z.
$$
Therefore $\mu((z\circ\T^n)^\sharp)$ is a Cauchy sequence in $\re$ and hence the limit
$$
\nu(z):=\lim_n \mu((z\circ\T^n)^\sharp)
$$
exists. By the Riesz representation theorem $\nu$ defines a
Borel probability measure in $\Sigma\times \Si$, and it is invariant because
$$
\nu(z\circ \T) =\lim_n\mu((z\circ \T^{n+1})^\sharp) = \nu(z).
$$

Now let $B:=L^*(L(A))-A$ and $\B:=B\circ\pi_1$. By formula~\eqref{eqB} we have that
$\B$ is a coboundary in $\Sigma\times\Si$.
Since $\pi_1\circ T^n= T^n$ we have that
\begin{align*}
0=\nu(\B) &= \lim_n\mu((\B\circ\T^n)^\sharp) \\
          &= \lim_n\mu(B\circ T^n) \\
          &= \mu(B).
\end{align*}
Since this holds for every $\mu\in\cM(T)$, by lemma~\ref{LCob}.\eqref{lcob.1},
$B\in\cB$ and then \linebreak
$(L^+\circ L)(A+\cB) \subset A+\cB$.

\end{proof}

\bigskip
\bigskip

\begin{thm}\label{MTP}\quad

 There is a residual subset $\cQ\subset\CS$ such that
 if $A\in\cQ$ and $A^*=L(A)$ then
\begin{equation}\label{eTT}
\begin{aligned}
\cM(A)=\{\mu\}\,&, \qquad \A(A)=\text{supp}(\mu), \\
\cM(A^*)=\{\mu^*\}\,&,\qquad  \A(A^*)=\text{supp}(\mu^*).
 \end{aligned}
\end{equation}
In particular
\begin{align*}
I_A(x)>0 \qquad&\text{ if}\qquad x\notin\text{supp}(\mu),\\
I_A(\om)>0 \qquad&\text{ if}\qquad \om\notin\text{supp}(\mu^*),
\end{align*}
and
\begin{align*}
R_A(x)>0 \qquad&\text{ if}\qquad x\notin\text{supp}(\mu) \quad\text{ and }\quad T(x)\in\text{supp}(\mu),\\
R_A(\om)>0 \qquad&\text{ if}\qquad \om\notin\text{supp}(\mu^*)\quad\text{ and }\quad \si(\om)\in\text{supp}(\mu).
\end{align*}

\end{thm}

\bigskip
\bigskip

\begin{proof}
 Observe that the subset $\cR$ defined in \eqref{eTR} in theorem~\ref{TR}
 is invariant under translations by coboundaries,
 i.e. $\cR=\cR+\cB$. Indeed  if $B=u\circ T-u\in\cB$, we have that
 \begin{gather*}
  \int (A+B)\,d\mu = \int A\,d\mu,\qquad \forall \mu\in\cB,\\
  S_{A+B}(x,y) = S_A(x,y) + B(x)-B(y), \qquad \forall x, y\in\Sigma.
  \end{gather*}
  Then the Aubry set  and the set of minimizing measures are
  unchanged: \linebreak
 $\cM(A+B)=\cM(A)$,\quad $\A(A+B)=\A(A)$.

  For the dynamical system $(\Si,\si)$ let
  $$
  \cR^*=\{\,\psi\in\CS\;|\;\cM(\psi)=\{\mu\},\;\A(\psi)=\text{supp}(\nu)\,\}
  $$
  By theorem~\ref{TR}, the subset $\cR^*$ contains a residual set in $\cF^*=\CS$ and it
  is invariant under translations by coboundaries: $\cR^*=\cR^*+\cB^*$

  By lemma~\ref{LLL} the linear map $L:\cF/\cB\to\cF^*/\cB^*$ is a homeomorphism
  with inverse $L^*$. Then the set $\cQ:=\cR\cap L^{-1}(\cR^*)=\cR\cap(L^*(\cR^*)+\cB)$
  contains a residual subset and satisfies~\eqref{eTT}.

  By Corollary~\ref{CPR} the other properties are automatically satisfied.

\end{proof}

\bigskip

From this last theorem it follows our main result about the generic potential $A$ to be in $\mathbb{G}$.

\end{document}